\documentclass{amsart}
\usepackage[utf8]{inputenc}
\usepackage{amsmath, tikz-cd, amssymb}
\usepackage{wasysym, stackengine, makebox, graphicx}
\usepackage{relsize}
\usepackage{geometry}
\usepackage{amsthm}
\usepackage{thmtools}
\usepackage{thm-restate}
\usepackage{hyperref}
\usepackage[capitalise, noabbrev]{cleveref}
\usepackage[colorinlistoftodos]{todonotes}

\def\Z{{\mathbb Z}}
\def\F{{\mathbb F}}

\def\N{{\mathbb N} } 
\def\C{{\mathcal C}}
\def\M{{\mathcal M}}
\def\D{{\mathcal D}}
\def\P{{\mathcal P}}

\def\E{{\mathbb E} }

\def\R{{\mathbb R}}

\def\smash{\widget}
\let\smash\temp
\let\wedge\vee

\newcommand\isom{\mathrel{\stackon[-0.1ex]{\makebox*{\scalebox{1.08}{\AC}}{=\hfill\llap{=}}}{{\AC}}}}

\newcommand\visom{\rotatebox[origin=cc] {90} {$ \isom $}}

\declaretheorem[numberwithin=section]{theorem}
\declaretheorem[sibling=theorem]{definition}
\declaretheorem[sibling=theorem]{corollary}

\declaretheorem[sibling=theorem]{lemma}
\declaretheorem[sibling=theorem]{proposition}
\declaretheorem[style = remark, sibling=theorem]{remark}
\def\n{ \langle n \rangle }

\def\0{ \langle 0 \rangle }
\def\1{ \langle 1 \rangle }
\def\2{ \langle 2 \rangle }
\title[On splitting the $BP\2$-cooperations algebra]{On a spectrum-level splitting of the $BP \langle 2 \rangle$-cooperations algebra}
\author{Elizabeth Ellen Tatum}
\address{Department of Mathematics, University of Illinios at Urbana-Champaign, Urbana, IL}
\email{etatum2@illinois.edu}

\newtheorem*{thm:intro:eiso}{\Cref{eiso}}
\newtheorem*{thm:intro:bakerLazarev}{\Cref{prop:bakerlazarev2}}
\newtheorem*{thm:intro:relASS}{\Cref{relASS}}
\newtheorem*{thm:intro:main}{\Cref{thm:main}}
\newtheorem*{thm:intro:liftif}{\Cref{liftif}}
\newtheorem*{thm:intro:leftiso}{\Cref{leftiso}}
\newtheorem*{thm:dominic:splitting2}{\Cref{thm:dominic:splitting2}}
\newtheorem*{thm:bp:split}{\Cref{thm:bp split}}
\newtheorem*{prop:zero:line}{\Cref{prop:zero_line}}
\newtheorem*{prop:iso}{\Cref{propiso}}
\newtheorem*{thm:intro:liftif2}{\Cref{liftif2}}
\newtheorem*{thm:intro:liftif3}{\Cref{liftif3}}
\begin{document}

\maketitle

\begin{abstract}
    In the 1980s, Mahowald and Kane used Brown-Gitler spectra to construct splittings of $bo \smash bo$ and $BP\1 \smash BP\1$. These splittings helped make it feasible to do computations using the $bo$- and $BP\1$-based Adams spectral sequences. In this paper, we construct an analogous splitting for $BP \langle 2 \rangle \smash BP\langle 2 \rangle$. 

\end{abstract}

\tableofcontents

\section{Introduction}
The main result of this paper is a splitting of $BP\2 \smash BP\2$ in terms of Brown-Gitler spectra (\cref{thm:main}). We will start by giving some context for this problem in \cref{motivation}. Then in \cref{outline}, we give a summary of the rest of the paper, and outline our construction of this splitting.  

\subsection{Motivation}\label{motivation}
\subsubsection{The n-truncated Brown-Peterson spectra}

Fix a prime $p$. The Brown-Peterson spectrum, denoted $BP$, is the complex-oriented cohomology theory associated to the universal $p$-typical formal group law. For each chromatic height $n$, there exists a closely related spectrum called the $n$-truncated Brown-Peterson spectrum, denoted $BP\n$. The spectrum $BP\n$ carries information about the formal group laws of height at most $n$. 

At low heights, these spectra have some other familiar names. The height zero spectrum $BP\0$ is the Eilenberg-Maclane spectrum for the integers localized at $p$, $H\Z_{(p)}$. The height one spectrum $BP\1$ is the connective $k$-theory spectrum $bu_{(2)}$ when $p=2$, and the Adams $l$-summand when $p$ is odd. The height two spectrum $BP\2$ carries information about topological modular forms. Specifically, $tmf_{1}(3)$ is an $E_{\infty}$-form of $BP\2$ at the prime $2$ \cite{lawson2012commutativity}. At the prime $3$, a spectrum of topological automorphic forms with a certain level structure is an $E_{\infty}$-form of $BP\2$ \cite{hill2010automorphic}. Although it is not known whether there are any $E_{\infty}$ structures for $BP\2$ at primes $p\ge 5$, the spectrum $BP\2$ is a summand of $tmf$, the connective spectrum for topological modular forms. 

One of the tools that can be used to extract information from $BP\n$ and related spectra about the homotopy groups of the sphere is called the $BP\n$-based Adams spectral sequence. The $BP\n$-based Adams spectral sequence is a member of a family of spectral sequences called the $R$-based Adams spectral sequences.

\subsubsection{The $R$-based Adams spectral sequence.} The Adams spectral sequence is a tool that can be used to approximate the homotopy groups of a connective spectrum $X$, such as the sphere. If $R$ is a ring spectrum satisfying certain conditions, then there is an $R$-based Adams spectral sequence 

\[ E_{1}^{*,*}(X,R) = \pi_{*}(R \smash R^{\bullet} \smash X) \Rightarrow \pi_{*}\widehat{X_{R}},\]
where $\widehat{X_{R}}$ denotes the $R$-completion of $X$.
If $R$ is a flat ring spectrum (eg, if $R_{*}R$ splits as a wedge sum of suspensions of $\pi_{*}R$ itself), then the $E_{2}$-page has a nice algebraic description, specifically $E_{2}^{*,*}(X,R) = Ext_{R_{*}R}(R_{*}, R_{*}X)$.

Both the classical Adams spectral sequence and the Adams-Novikov Spectral Sequence are examples of such spectral sequences, for $R = H\F_{p}$ and $R = MU$, respectively. Although their $E_{2}$-pages are algebraic objects, they are still very complex and do not have a closed-form description. Furthermore, it becomes very difficult to compute differentials when $X = S^{0}$. Despite the difficulty of computing those differentials, these two spectral sequences are essential tools for studying the homotopy groups of the spheres. (A list of references and computations can be found in \cite{green_book}). More recently, Isaksen, Wang, and Xu have extended those computations through the $90$-stem \cite{isaksen_wang_xu} at the prime $2$. 

\subsubsection{Background on $v_{2}$-periodicity }
Another approach is to study these homotopy groups one height at a time. The $K(n)$-local homotopy groups of the sphere give an approximation of $\pi_{*}S$ containing information about $v_{n}$-periodicity. At height $2$, this approach has been well-developed. In \cite{shimomura_wang}, Shimomura and Wang computed the homotopy groups of the $K(2)$-local sphere at the prime $p=3$. In \cite{ghmr}, Goerss, Henn, Mahowald, and Rezk improved the conceptual framework and organization of these computations by constructing a resolution of the $K(2)$-local sphere at the prime $3$, and explicitly describing the homotopy groups of its fibers. In \cite{ShimomuraYabe}, Shimomura and Yabe computed the homotopy groups of the $E(2)$-local sphere for primes $p >3$. In \cite{behrens2012homotopy}, Behrens gave a more conceptual description of these groups, and used this description to compute the $K(2)$-local homotopy groups of the sphere at these primes. These $K(2)$-local homotopy groups provide a partial understanding of $v_{2}$-periodicity, but more is left to understand about the $v_{2}$-periodic structure of $\pi_{*}S$, and its interactions with the $v_{1}$ and $v_{0}$-periodicity. The $BP\2$-based Adams spectral sequence is one tool that could potentially be used to improve this understanding.

\subsubsection{The $bo$-based and $BP\n$-based Adams spectral sequences}
Consider the $BP\n$-based Adams spectral sequence \[E_{r}(X, BP\n) \Longrightarrow \pi_{*}\widehat{X_{BP\n}} .\] Note that $\widehat{X_{BP\2}}$ is actually equivalent to the $p$-localization $\widehat{X_{(p)}}$ \cite[Thm.~2.2.13]{green_book}. The $BP\langle 1 \rangle$-based Adams spectral sequence highlights the $v_{1}$-periodicity and makes it easier to extract information about the $v_{1}$-periodic structure. It is expected that the $BP\langle 2 \rangle$-based Adams spectral sequence will display $v_{2}$-periodicity in a similar way. Throughout the rest of this paper, we will assume that everything is $p$-completed.

Initially, using a $BP\n$-based Adams spectral sequence rather than the classical or Adams-Novikov spectral sequence might not seem like such a good idea- the spectra $BP\n$ are not flat, and computing the $E_{2}$-page is very difficult. However, Mahowald's work on the $bo$-based Adams spectral sequence \cite{mahowald_splitting} at the prime $p=2$ indicated that such spectral sequences could be very useful, despite the difficulty in finding their $E_{2}$-pages. The spectrum $bo$ is also not flat, and describing the $E_{2}$-page of the $bo$-based Adams spectral sequence is very difficult. However, Mahowald showed that, in a certain range, the $bo$-based Adams spectral sequence collapses at the $E_{2}$-page. These computations gave a great deal of insight into the $v_{1}$-periodic structure of the homotopy groups of the sphere, and were sufficient to prove the Telescope Conjecture at height one for the prime $2$. Gonzalez showed in \cite{gonzalez} that the $BP\langle 1 \rangle$-based Adams spectral sequence behaves in a similar way at odd primes, making it a useful tool for describing $v_{1}$-periodicity at odd primes. 

In further work, Lellman and Mahowald computed the $v_{1}$-periodic portion of the $E_{2}$-page of the $bo$-based Adams spectral sequence \cite{lellmann_mahowald}. In \cite{donald_davis}, Davis computed the $v_{1}$-torsion piece of the $bo$-based Adams spectral sequence through the $20$-stem. Combined with Lellman and Mahowald's computations, this gave a complete description of the $bo$-based Adams spectral sequence through the $20$-stem. Recent work by Beaudry-Behrens-Bhattacharya-Culver-Xu \cite{bbbcx_paper} has extended these computations through the $40$-stem. 

In order to use a $BP\n$-based Adams spectral sequence, we need a good understanding of $BP\langle n \rangle\smash BP\n$. One way to do this is to split $BP\langle n \rangle\smash BP\n$ into more manageable pieces, specifically as a sum of finitely generated $BP\n$-modules. These splittings were essential for the computations described above. This is where Brown-Gitler spectra come into the picture.\\

\subsubsection{Brown-Gitler spectra.} In the 1970's, Brown and Gitler constructed a family of spectra realizing certain sub-comodules of the dual Steenrod algebra at the prime $p=2$ \cite{brown_gitler}. Cohen then constructed the corresponding family for odd primes. Subsequently, Cohen \cite{Cohen}, Goerss-Jones-Mahowald \cite{goerss_jones_mahowald}, Klippenstein \cite{klippenstein}, and Shimamoto \cite{shimamato} constructed finite spectra realizing the analogous family of sub-comodules of $H_{*}BP\n$, for $0 \le n \le 2$. These spectra are collectively known as generalized Brown-Gitler spectra. The original Brown-Gitler spectra were constructed for studying immersions of manifolds, but these spectra and their generalizations have been used for many other interesting applications in the years since. One application is to produce decompositions of $BP\langle n \rangle\smash BP\n$, as well as smash products of other related spectra, into simpler pieces.

We use a filtration on the dual Steenrod algebra $A_{*}$, known as the weight filtration and denoted $wt$, to define the Brown-Gitler sub-comodules of the dual Steenrod algebra. At the prime $2$, the Steenrod algebra has the form \[A_{*} \cong \F_{p}[\bar{\xi}_{1}, \bar{\xi}_{2}, \ldots]. \] At odd primes, the dual Steenrod algebra has the form
\[A_{*} \cong \F_{p}[\bar{\xi}_{1}, \bar{\xi}_{2}, \ldots] \otimes E(\bar{\tau}_{0}, \bar{\tau}_{1}, \ldots). \] At any prime, we can define the weight filtration, by setting $wt(\bar{\xi}_{i}) = wt(\bar{\tau}_{i}) = p^{i}$ for all $i$, and $wt(xy) = wt(x)wt(y)$. The comultiplication $\Delta: A_{*} \rightarrow A_{*} \otimes A_{*}$ does not increase the weight.

The Brown-Gitler comodule $B_{0}(k)$ is defined to be the sub-comodule of $A_{*}$ generated by monomials of weight at most $k$. Since $H_{*}BP\n$ is a comodule subalgebra of $A_{*}$, the weight filtration extends to $H_{*}BP\n$ for each $n$. Similarly, there is a family of height $n$ Brown-Gitler comodules $\{B_{n}(k)|k \in \N\}$, where $B_{n}(k)$ is the sub-comodule of $H_{*}BP\n$ generated by monomials of weight at most $k$. (Note that in this choice of notation, the Brown-Gitler comodules $B_{n}(pk), \ldots, B_{n}(pk + p-1)$ are the same for a fixed $n \ge 1$). 

The original Brown-Gitler spectra, which we will denote $H_{k}$, are finite spectra such that $H_{k} = B_{-1}(k)$. These spectra were constructed by Brown-Gitler at the prime $p=2$ \cite{brown_gitler} and by Cohen at odd primes \cite{Cohen}. In \cite{shimamato}, Shimamato constructed the integral Brown-Gitler spectra $\{H\Z_{k}\}$, finite spectra such that $H_{*}H\Z_{k} \cong B_{0}(k)$. Goerss-Jones-Mahowald constructed the height one Brown-Gitler spectra in \cite{goerss_jones_mahowald}. The height one Brown-Gitler spectra include the family of finite spectra $\{l_{k}\}$ such that $H_{*}l_{k} \cong B_{1}(k)$. At the prime $2$, there is another family of Brown-Gitler spectra at height one, the $bo$-Brown Gitler spectra. The homology $H_{*}bo$ is also a subalgebra of the dual Steenrod algebra, and so the weight filtration can be applied to $H_{*}bo$. The finite spectrum $bo_{k}$ realizes the weight $\le k$ component of $H_{*}bo$. In \cite{klippenstein}, Klippenstein constructed a family of finite spectra $\{ BP\2_{k} \}$, realizing the comodules $\{B_{2}(k)\}$. It is not known whether spectra realizing $B_{n}(k)$ exist for higher heights $n \ge 3$.\\ 

\subsubsection{Splittings}
At each prime $p$ and each height $n\ge 0$, there exists an isomorphism of $A_{*}$
-comodules 

\begin{equation}\label{eqn6} H_{*}BP\langle n \rangle\otimes H_{*}BP\langle n \rangle \cong \bigoplus_{k=0}^{\infty} H_{*}BP\langle n \rangle\otimes \Sigma^{qk}B_{n-1}(k)   .\end{equation}

Analogous isomorphisms exist for $bo$ and $tmf$ at the prime $2$, specifically

\begin{equation} H_{*}bo \otimes H_{*}bo \cong \bigoplus_{k=0}^{\infty} H_{*}bo \otimes \Sigma^{qk}B_{0}(k) \end{equation}

\begin{equation}\label{tmf} H_{*}tmf \otimes H_{*}tmf  \cong \bigoplus_{k=0}^{\infty} H_{*}tmf \otimes \Sigma^{qk}B_{1}(k) .\end{equation}

At low heights, the Brown-Gitler spectra have been used to realize these splittings. It is well-known that there is a splitting \begin{equation*}\label{eqn1} BP\langle 0 \rangle \smash BP\langle 0 \rangle \simeq \vee_{k=0}^{\infty}BP\langle 0 \rangle \smash \Sigma^{qk}H_{k} ,\end{equation*} where $q = 2(p-1)$. In the 1980's, Mahowald constructed the following splitting of $bo \smash bo$ in terms of integral Brown-Gitler spectra \cite[Thm~2.4]{mahowald_splitting}:

\begin{equation}\label{eqn2} bo \smash bo \simeq \bigvee_{k=0}^{\infty} bo \smash \Sigma^{qk}H\Z_{k} .\end{equation}

Subsequently, Kane used Mahowald's methods to construct an analogous splitting for $BP\langle 1 \rangle \smash BP\langle 1 \rangle$ at odd primes \cite[Thm~11.1]{kane}: 

\begin{equation}\label{eqn3} BP\langle 1 \rangle \smash BP\langle 1 \rangle \simeq \bigvee_{k=0}^{\infty} BP\langle 1 \rangle \smash \Sigma^{qk}H\Z_{k} .\end{equation}

This might lead one to guess that all of these splittings are realizable when the relevant Brown-Gitler spectra exist. However, things get more interesting at height $2$: Davis and Mahowald demonstrated that the isomorphism for $tmf$ (\ref{tmf}) is not realizable, by showing that the homotopy groups of $tmf \smash tmf$ and $\bigvee_{k=0}^{\infty}\Sigma^{qk}tmf \smash bo_{k} $ are not isomorphic \cite{davis_mahowald}. In spite of this, it was still conjectured that at the height $n=2$, the splitting (\ref{eqn6}) could be realized for all primes. 

In particular, Culver provided strong evidence towards the $2$-primary version of this conjecture in \cite{dominic_even}, as well as the odd-primary version in \cite{dominic_odd}. He showed that $BP\langle 2 \rangle \smash BP\langle 2 \rangle$ splits as 

\begin{equation}\label{eqn4} BP\langle 2 \rangle \smash BP\langle 2 \rangle \simeq C \vee V ,\end{equation}

\noindent where $V$ is a wedge sum of Eilenberg-Maclane spectra realizing $H\F_{p}$, and $\pi_{*}C$ is $v_{2}$-torsion free and concentrated in even degrees \cite{dominic_even, dominic_odd}. Furthermore, he used this splitting to show the following isomorphism of homotopy groups: 

\begin{equation}\label{eqn5} \pi_{*}\Big(BP\langle 2 \rangle \smash BP\langle 2 \rangle\Big) \cong \pi_{*}\Big(\bigvee_{k=0}^{\infty}\Sigma^{qk}BP\langle 2 \rangle \smash l_{k}\Big). \end{equation}

The main result of this paper is that we can indeed realize the splitting (\ref{eqn5}).

\begin{thm:intro:main}
At all primes $p$, there exists a splitting (up to $p$-completion)
\[BP\langle 2 \rangle \smash BP\langle 2 \rangle \simeq \bigvee_{k=0}^{\infty} \Sigma^{qk}BP\langle 2 \rangle \smash l_{k}.\]
\end{thm:intro:main}

\subsection{Summary of the paper}\label{outline}

Now we will outline our approach to the construction of the splitting of \cref{thm:main}. A very brief summary is as follows. In \cref{background}, we start by recalling some helpful algebraic results. In \cref{sec:bakerLazarev}, we recall Baker-Lazarev's relative Adams spectral sequence in the category of $BP\2$-modules, and then we reduce the problem of constructing the splitting of \cref{thm:main} to showing that a certain family of classes survive this spectral sequence. In \cref{spec:aux}, we discuss a universal coefficient spectral sequence and a hypercohomology spectral sequence. In \cref{square}, we construct a square of spectral sequences relating these auxiliary spectral sequences to the relative Adams spectral sequence. We then use this square in \cref{sec:analysis:po} to show that the necessary classes do indeed survive the relative Adams spectral sequence, proving \cref{thm:main}. The remainder of this section gives a more detailed outline.

We start by recalling some well-known algebra results in \cref{sec:equiv} that will be useful throughout this paper. In \cref{sec:BPn}, we recall the homology of $BP\n$. In \cref{maps}, we recall the following splitting of $H_{*}BP\n$ in terms of Brown-Gitler comodules.

\begin{thm:intro:eiso}\cite[Lemma~4.8]{dominic_odd}\cite[Prop~3.3]{dominic_even}
Let $n \ge 0$. Then there exists a family of maps \[\{ \theta_{k}: \Sigma^{qk}B_{n-1}(k) \rightarrow H_{*}BP\n |\ k \in \N \}\] such that their sum 
\[\bigoplus\limits_{k=0}^{\infty} \theta_{k}:\bigoplus\limits_{k=0}^{\infty}\Sigma^{qk}B_{n-1}(k) \rightarrow H_{*}BP\n    \]
is an isomorphism of $E(n)_{*}$-comodules.
\end{thm:intro:eiso}

Note that this is only an isomorphism of $E(n)_{*}$-comodules (not an $A_{*}$-comodule isomorphism). So we cannot realize this exact decomposition on the level of spectra. However, at low heights, we can use Baker-Lazarev's relative Adams spectral sequence to leverage this decomposition to the spectrum-level splitting of \cref{thm:main} at the height $n=2$.  

Before discussing the relative Adams spectral sequence, we begin \cref{sec:BP2mod} by recalling the category of $R$-module spectra as discussed in \cite[Section~III.3]{ekmm}. Let $R$ be an $S$-algebra ($A_{\infty}$ ring spectrum). Let $M$ be a right $R$-module, and let $N$ be a left $R$-module. We let $M \smash_{R} N$ denote the coequalizer \cite[Def~III.3.1]{ekmm}
\[
\begin{tikzcd}
{M  \smash R  \smash N} \arrow[r, shift right] \arrow[r, shift left] & {M \smash N}  \arrow[r] & {M \smash_{R} N}.\\
\end{tikzcd}
\]

Let $[M,N]^{R}$ denote the group of homotopy classes of $R$-module maps from $M$ to $N$, and let $H_{*}^{R}M = \pi_{*}(H \smash_{R} M)$. In \cite{bakerlazarev}, Baker-Lazarev introduce the following Adams spectral sequence in the category of $R$-modules:

\begin{thm:intro:bakerLazarev}\cite[Prop~2.1]{bakerlazarev}
    Let $L,M$ be $R$-modules, and let $E$ be a commutative ring spectrum with $E_{*}^{R}E$ flat as a left or right $E_{*}$-module. If $E_{*}^{R}L$ is projective as an $E_{*}$-module, then there is an Adams spectral sequence with \[ E_{2}^{s,t} = Ext_{E_{*}^{R}E}^{s,t}(E_{*}^{R}L, E_{*}^{R}M).\] 
\end{thm:intro:bakerLazarev}

In the case $R = BP\2$, $E = H$, $M = BP\2 \smash l_{k}$, and $N = BP\2 \smash BP\2$, we can use the other results of \cref{sec:BP2mod} to show that this spectral sequence is strongly convergent, and can be written in the following form. 

\begin{thm:intro:relASS}\label{intro:relASS}
For all $k \in \N$, there exists a strongly convergent relative Adams spectral sequence
\[ {}^{ASS}E_{2}^{s,t} = Ext_{E(2)_{*}}(H_{*}\Sigma^{qk}l_{k}, H_{*}BP\2) \Longrightarrow [BP\2 \smash \Sigma^{qk}l_{k}, BP\2 \smash BP\2]^{BP\2}. \]
\end{thm:intro:relASS}

In \cref{sec:ass}, we explain how we will use this Adams spectral sequence to lift the $E(2)$-module splitting of \cref{eiso} to the spectrum-level splitting of \cref{thm:main}. Note that we can think of the map \[\theta_{k}: H_{*}l_{k} \rightarrow H_{*}BP\2 \smash BP\2\] as a class in ${}^{ASS}E_{2}^{0,0}$. To lift the map, we need to show that $\theta_{k}$ survives the spectral sequence. For dimension reasons, it is impossible for any differential to hit $\theta_{k}$. So if we can show that the differential leaving $\theta_{k}$ is always zero, then we will be able to conclude that $\theta_{k}$ survives the spectral sequence. Note that the Adams differential $d_{r}$ has degree $(r,r-1)$. We call classes in degree $E_{2}^{r,r-1}$ the \textit{potential obstructions} to lifting $\theta_{k}: H_{*}B_{n}(k) \rightarrow H_{*}BP\n$. 

Suppose that for each $k$, the map $\theta_{k}$ does indeed survive the relative Adams spectral sequence. Then we will have constructed a family of $BP\2$-module maps 

\[\widetilde{\theta_{k}}: BP\2 \smash l_{k} \rightarrow BP\2 \smash BP\2\] 

whose sum induces the isomorphism $\bigoplus_{k=0}^{\infty}\theta_{k}$ on $H_{*}^{BP\2}$-homology. Then it will follow immediately from \cref{whitehead} that \[\bigvee\widetilde{\varphi_{k}}: \bigvee\limits_{k=0}^{\infty} BP\2 \smash l_{k} \rightarrow BP\2 \smash BP\2 \]
is a homotopy equivalence up to $p$-completion. 

It is interesting to note that at height $n=1$, there are no potential obstructions to lifting the analogous maps $\theta_{k}: H_{*}B_{0}(k) \rightarrow H_{*}BP\1 $ to maps $\widetilde{\theta_{k} }: BP\1 \smash B_{0}(k) \rightarrow BP\1 \smash BP\1$. This recovers Kane's splitting (\ref{eqn5}). However, at height $2$, there are many potential obstructions, and we will spend the remainder of the paper demonstrating that they are not actual obstructions. (There is an analogous spectral sequence to \cref{main:relASS} at height $n=3$. However, we do not yet know whether the potential obstructions at that height survive the spectral sequence. At height $n >3$, it is not yet known whether the necessary Brown-Gitler spectra exist). 

In \cref{sec:ucss} we will introduce Robinson's universal coefficient spectral sequence. We will use this auxiliary spectral sequence to show that the potential obstructions survive the Adams spectral sequence. This is an adaptation of the technique used by Klippenstein to construct a splitting of $bu \smash BP\n$ in terms of finitely generated $bu$-modules (\ref{klippenstein}).

Specifically, we will compare the \textit{universal coefficient spectral sequence}

\begin{equation}\label{intro:ucss}
E_{2}^{u,v} = Ext_{BP\2_{*}}\big( BP\2_{*}l_{k}, BP\2_{*}BP\2 \big) \Longrightarrow [l_{k}, BP\2 \smash BP\2 ] 
\end{equation}

to the relative Adams spectral sequence 
\begin{equation}\label{main:relASS} E_{2}^{s,t} = Ext_{E(2)_{*}}(H_{*}l_{k}, H_{*}BP\2) \Longrightarrow [BP\2 \smash l_{k}, BP\2 \smash BP\2]^{BP\2}\end{equation}
to show that the potential obstructions survive the relative Adams spectral sequence. Klippenstein explicitly computed the $E_{2}$-pages of the universal coefficient spectral sequence and the Adams spectral sequences needed to construct the splitting of \cref{klippenstein}. However, both the $E_{2}$-pages of the Adams spectral sequence (\ref{main:relASS}) and the universal coefficient spectral sequence (\ref{intro:ucss}) are more complicated to compute than their height one analogues. So instead of completely computing these $Ext$ terms and directly comparing them, we will use another spectral sequence called the \textit{hypercohomology spectral sequence} as a bridge between them. 

In \cref{sec:hss}, we will review Cartan-Eilenberg's hypercohomology spectral sequence. Let $G(-)$ denote the right-derived functor $G(-) = \R Hom_{E(2)_{*}}(\F_{p}, -)$, and let $\P(2) = G(\F_{p})$. We will use the equivalence of categories of \cref{cor:1:equiv} to present the following special case of the universal coefficient spectral sequence 

\begin{equation}\label{intro:kmhss}
\begin{split}\bigoplus \limits_{ \substack{r= r_{2} - r_{1} \\ t = t_{2} - t_{1} } }Ext_{\P(2)}^{q}( \big( Ext_{E(2)_{*}}^{r_{1}, t_{1}}(\F_{p}, H_{*}l_{k}), Ext_{E(2)_{*}}^{r_{2}, t_{2}}(\F_{p}, H_{*}BP\2) \big)\\
\Longrightarrow H^{ r + q, t}\R Hom_{\P(2)}\big( G(H_{*}l_{k}), G(H_{*}BP\2 ) \big).
\end{split}
\end{equation}

Then we will use Koszul duality to show the following corollary.

\begin{restatable*}[]{corollary}{mhss}
There exists an isomorphism
\[ H^{ u, t}\R Hom_{\P(2)}\big( G(H_{*}l_{k}), G(H_{*}BP\2) \big) \cong Ext_{E(2)_{*}}^{ u, t }\big(H_{*}l_{k}, H_{*}BP\2 \big). \]
\end{restatable*}

So we can think of this particular hypercohomology spectral sequence as one that converges to the $E_{2}$-page of the Adams spectral sequence.

We begin \cref{sec:dominic:splitting} by recalling Culver's splitting of $BP\2 \smash BP\2$ (\ref{eqn4}). It follows immediately from his analysis that $BP\2 \smash l_{k}$ splits in an analogous way, so we can state the following theorem. 
\begin{thm:dominic:splitting2}
If $X = BP\2$ or $X = l_{k}$, then then there exists a splitting 

\[ BP\2 \smash X \simeq C_{X} \vee V_{X},\]

such that $V_{X}$ is a sum of suspensions of $\F_{p}$-Eilenberg-Maclane spectra, and $Ext_{A_{*}}(\F_{p}, H_{*}C_{X})$ is $v_{2}$-torsion free and concentrated in even $(t-s)$-degrees. This splitting is unique up to equivalence.
\end{thm:dominic:splitting2}

We will use a relative Adams spectral sequence to upgrade this to a $BP\2$-module splitting.

\begin{thm:bp:split}   
The splitting of \cref{thm:dominic:splitting2} is a $BP\2$-module splitting. 
\end{thm:bp:split}

{In \cref{sec:leftiso}, we will use this splitting to extend Culver's analysis \cite{dominic_odd} of the Adams spectral sequence converging to $BP\2_{*}BP\2$. In \cref{sec:35}, we will use this analysis to show that the $E_{2}$-page of the hypercohomology spectral sequence (\ref{intro:kmhss}) and the $E_{2}$-page of the universal coefficient spectral sequence (\ref{intro:ucss}) are isomorphic, as stated in the following lemma. }

\begin{thm:intro:leftiso}
There exists an isomorphism
\[ Ext_{\P(2)}^{u,*}\big(  Ext_{E(2)_{*}} ( \F_{p}, H_{*}l_{k}), Ext_{E(2)_{*}} ( \F_{p}, H_{*}BP\2) \big) \cong Ext_{BP\2_{*}}^{u,*}\big(BP\2_{*}l_{k}, BP\2_{*}BP\2  \big) .\]

\end{thm:intro:leftiso}

We can express these two results as the following square relating the Adams spectral sequence and the universal coefficient spectral sequence. (This square is not necessarily commutative, but that is not an issue- it will still provide sufficient information about the relationship between the Adams spectral sequence (\ref{main:relASS}) and the universal coefficient spectral sequence (\ref{intro:ucss}) for our purposes.) 

\begin{equation}
\begin{tikzcd}[ampersand replacement=\&]
\bigoplus\limits_{\substack{r=r_{2}-r_{1} \\ t = t_{2} - t_{1}} }Ext_{\P(2)}^{u}\big( Ext_{E(2)_{*}}^{r_{1}, t_{1}} ( \F_{p}, H_{*}l_{k}), Ext_{E(2)_{*}}^{r_{2}, t_{2}} ( \F_{p}, H_{*}BP\2) \big) \arrow[Rightarrow, r, "HSS"] \arrow[d, phantom, "\visom"]
\&
Ext_{E(2)}^{r+u,t}\big(H_{*}l_{k}, H_{*}BP\2  \big)\arrow[d, Rightarrow, "ASS"]\\
Ext_{BP\2_{*}}^{u,t-r}\big(BP\2_{*}l_{k}, BP\2_{*} BP\2 \big)\arrow[r, Rightarrow, "UCSS"]
\&
\left[  l_{k}, BP\2 \smash BP\2 \right]_{t-r-u} \\
\end{tikzcd}
\end{equation}

In \cref{sec:analysis:po}, we will use the square to analyze the potential obstructions. We will start by using Culver's splitting from \cref{sec:dominic:splitting} to show that we can actually look at a smaller square: Let $C$, $V$ denote the summands $C_{BP\2}, V_{BP\2}$ and let $C_{k}$, $V_{k}$ denote the summands $C_{l_{k}}, V_{l_{k}}$ in the splitting of \cref{thm:dominic:splitting2}. We can use these $BP\2$-module splittings to decompose the $E_{2}$-page of the Adams spectral sequence (\ref{main:relASS}) as as a sum of four components. For degree reasons, any potential obstructions must be contained in the summand $Ext_{E(2)_{*}}(H_{*}^{BP\2}C_{k}, H_{*}^{BP\2}C)$. So we will restrict our attention to the following square. 

\begin{equation}
    \begin{tikzcd}[ampersand replacement=\&]\label{intro:smallsquare}
Ext_{\P(2)}\big( Ext_{E(2)_{*}} ( \F_{p}, H_{*}^{BP\2}C_{k}), Ext_{E(2)_{*}} ( \F_{p}, H_{*}^{BP\2}C) \big) \arrow[Rightarrow, r, "HSS"] \arrow[d, phantom, "\visom"]
\&
Ext_{E(2)}\big(H_{*}^{BP\2}C_{k}, H_{*}^{BP\2}C  \big)\arrow[Rightarrow, d, "ASS"]\\
Ext_{BP\2_{*}}\big(\pi_{*}C_{k}, \pi_{*}C \big)\arrow[Rightarrow, r, "UCSS"]
\&
\left[C_{k}, C\right]^{BP\2}\\
\end{tikzcd}
\end{equation}

We start to analyze this square by identifying a vanishing line in the hypercohomology spectral sequence. 

\begin{prop:zero:line}
For all $u > 2$,
\[  Ext_{\P(2)}^{u,*}\big( Ext_{E(2)_{*}} ( \F_{p}, H_{*}^{BP\2}C_{k}), Ext_{E(2)_{*}} ( \F_{p}, H_{*}^{BP\2}C) \big) = 0. \]
\end{prop:zero:line}

It follows that every class on the line $Ext_{\P(2)}^{u=1,*}$ survives the universal coefficient spectral sequence. We can combine this observation with the isomorphism between the $E_{2}$-pages of the hypercohomology spectral sequence and the universal coefficient spectral sequence (\cref{leftiso}) to prove the following.

\begin{prop:iso}
The odd $(t-s)$-degree component of the $E_{2}$-page of the Adams spectral sequence is isomorphic to the $(u=1)$-line of the universal coefficient spectral sequence, that is,
\[ \bigoplus\limits_{t-s \text{ odd}} Ext_{E(2)_{*}}^{s,t} \big(  H_{*}^{BP\2}C_{k}, H_{*}^{BP\2}C \big) \cong Ext_{BP\2_{*}}^{u=1,*}(\pi_{*}C_{k}, \pi_{*}C). \]
\end{prop:iso}

\cref{prop:zero_line} and the isomorphism of \cref{leftiso} imply that $Ext_{BP\2_{*}}^{u=1,*}$ must survive the universal coefficient spectral sequence. Combined with \cref{propiso}, this implies the following theorem. 

\begin{restatable*}[]{theorem}{oddsurvives}
Let $x \in   \bigoplus\limits_{t-s \text{ odd}} Ext_{E(2)_{*}}^{s,t} \big(  H_{*}^{BP\2}C_{k}, H_{*}^{BP\2}C \big)$. Then $x$ survives the Adams spectral sequence \[Ext_{E(2)_{*}}^{*,*} \big(  H_{*}^{BP\2}C_{k}, H_{*}^{BP\2}C \big) \Longrightarrow [C_{k}, C]^{BP\2}.\]
\end{restatable*}

Note that the potential obstructions to lifting $\theta_{k}$ are all contained in odd $(t-s)$-degree, so the following corollary is immediate. 
\begin{restatable*}[]{corollary}{posurvives}
Let $x$ be a potential obstruction to lifting $\theta_{k}$ in the Adams spectral sequence 
\[ {}^{ASS}E_{2}^{s,t} = Ext_{E(2)_{*}}^{s,t}\big( H_{*}l_{k}, H_{*}BP\2 \big) \Longrightarrow [BP\2 \smash l_{k}, BP\2 \smash BP\2 ]^{BP\2}.\]
Then $x$ is not a boundary, that is, there is no $y \in {}^{ASS}E_{r}$ such that $d_{r}(y) = x$.
\end{restatable*}

So indeed the maps $\theta_{k}$ can be lifted, and we can construct our splitting.

\subsection*{Acknowledgments}
I would like to thank my thesis advisor, Vesna Stojanoska, for suggesting this problem for my thesis, and for her guidance and support throughout this project. I am also grateful to the other members of my thesis committee, Charles Rezk, Matt Ando, and Jeremiah Heller, for helpful conversations and feedback. Thank you to Dominic Culver for many valuable discussions about $BP\2$-cooperations. Thank you also to Dylan Wilson, for a very informative email about Koszul duality, and to Paul VanKoughnett and Noah Riggenbach for helpful conversations about $\E_{n}$-commutative algebras.

\section{Background}\label{background}
In \cref{sec:equiv}, we will start by reviewing some basic facts about comodules over finite graded coalgebras which will be used throughout this paper. In \cref{sec:BPn}, we will recall the homology of $BP\n$. In \cref{maps}, we will discuss the well-known $E(n)_{*}$-comodule splitting \[H_{*}BP\n \cong \bigoplus \limits_{k=0}^{\infty} \Sigma^{qk} B_{n}(k).\]

\subsection{Equivalence of categories}\label{sec:equiv}

Let $\Gamma$ be a graded coalgebra over a field $k$, and let $\Gamma^{*}$ denote its dual graded algebra. Let $Comod_{\Gamma}$ denote the category of graded left $\Gamma$-comodules, and let $Mod_{\Gamma^{*}}$ denote the category of graded left $\Gamma^{*}$-modules. We will use the following functor from $Comod_{\Gamma}$ to $Mod_{\Gamma^{*}}$.

\begin{lemma} \cite[Lemma~0.3.3]{palmieri2001stable}\label{lem:1:montgomery1}
Let $\Gamma$ be a graded coalgebra over a field $k$, and let $\Gamma^{*}$ be its dual graded algebra. 
Let $\langle, \rangle: \Gamma^{*} \otimes \Gamma \rightarrow k$ be the evaluation map. That is, if $\alpha \in \Gamma^{*}$ and $x \in \Gamma$, then $\langle \alpha, x \rangle = \alpha(x) $. Let $(M, \psi: M \rightarrow \Gamma \otimes M)$ be a graded left $\Gamma$-comodule. Let $\mu_{M}: \Gamma^{*} \otimes M \rightarrow M$ be the composition

\[
\begin{tikzcd}[ampersand replacement=\&]
\mu_{M}: \ \Gamma^{*} \otimes M \arrow[r, " 1 \otimes \psi"]
\&
\Gamma^{*} \otimes \Gamma \otimes M \arrow[r, "{\langle ,\rangle  \otimes 1}"]
\& 
k \otimes M \cong M \\
\end{tikzcd}
.\]
Then $(M, \mu_{M})$ is a graded left $\Gamma^{*}$-module. Furthermore, the assignment $(M, \psi_{M}) \mapsto (M, \mu_{M})$ defines a functor $\varphi: Comod_{\Gamma} \rightarrow Mod_{\Gamma^{*}} $.

\end{lemma}

When $\Gamma$ is finite, the functor $\varphi$ is an equivalence of categories.

\begin{corollary}\cite[Lemma~0.3.3]{palmieri2001stable}\label{cor:1:equiv}
Suppose that $\Gamma$ is a finite graded coalgebra over a field $k$. Then $\varphi: Comod_{\Gamma} \rightarrow Mod_{\Gamma^{*}} $ is an equivalence of categories. 
\end{corollary}

We are particularly interested in the case where $\Gamma = E(2)_{*}:= E(\bar{\tau}_{0}, \bar{\tau}_{1},  \bar{\tau}_{2})$, and $\Gamma^{*} \cong E(2):= E(Q_{0}, Q_{1},  Q_{2})$. The equivalence of categories between $E(2)$-modules and $E(2)_{*}$-comodules will be used throughout the rest of the paper.\\

We will also use the following results relating projective, free, and injective modules and comodules. The first is immediate from the equivalence of categories.

\begin{lemma}\cite[Prop~4]{doi}\label{doi}
A module $M$ is projective(resp. injective) as an $E(2)_{*}$-comodule if and only if $M$ is projective (resp. injective) as an $E(2)_{*}$-comodule.
\end{lemma}

\begin{lemma}\cite[Prop.~12.8]{margolis_book}\label{free:proj:inj} 
Let $M$ be a module over $E(2)$ which is finitely generated in each degree. Then the following are equivalent:
\begin{enumerate}
\item $M$ is free;
\item $M$ is projective;
\item $M$ is injective.
\end{enumerate}
\end{lemma}

The following corollary is immediate. 

\begin{corollary}
Let $M,N$ be $E(2)_{*}$-comodules. Then
\[ Ext_{E(2)_{*}}^{s,t}(M,N) \cong Ext_{E(2)}^{s,t}(M,N).\]
\end{corollary} 

\begin{proof}
Consider a projective resolution

\[ 0 \leftarrow M \leftarrow P^{0} \leftarrow P^{1} \leftarrow \cdots \]
of $E(2)_{*}$-comodules. By \cref{doi}, each $P^{i}$ is also a projective $E(2)$-module. By the equivalence of categories,
\[Hom_{E(2)_{*}}(P^{i}, N) \cong Hom_{E(2)}(P^{i}, N) .\] 
So indeed $Ext_{E(2)_{*}}^{s,t}(M,N) \cong Ext_{E(2)}^{s,t}(M,N)$.
\end{proof}

\subsection{The action of $E(n)$ on $H_{*}BP\n$}\label{sec:BPn}
Let $A_{*}$ denote the dual Steeenrod algebra, ie $A_{*} = H_{*}H$. Let $Q_{i}$ denote the $i^{th}$ Milnor primitive, and let $E(n)$ denote the subalgebra $E(Q_{0}, Q_{1}, \ldots, Q_{n})$ of the Steenrod algebra. 

We will start by reviewing the action of $E(n) = E(Q_{0}, Q_{1}, \ldots Q_{n})$ on $A//E(n)_{*}$. In \cref{action:odd}, we explicitly describe the action at odd primes, and in \cref{action:even}, we explicitly describe the action at $p=2$. Then we will recall the well-known fact that $H_{*}BP\n \cong A//E(n)_{*}$ at all primes. 

\subsubsection{The odd-primary case}\label{action:odd}
At odd primes,
\[ A_{*} \cong \F_{p}[\xi_{1}, \xi_{2}, \ldots ] \otimes E(\tau_{0}, \tau_{1}, \ldots), \] 

where $|\bar{\xi}_{i}| = 2p^{i} - 2$ and $|\bar{\tau}_{i}| = 2p^{i} - 1$.

The homology of $BP\n$ is most easily described using the conjugates of these usual generators. Let $\bar{\xi}_{k}$ denote the conjugate of $\xi_{k}$, and let $\bar\tau_{k}$ denote the conjugate of $\tau_{k}$. 

Let $\Delta: A_{*} \rightarrow A_{*} \otimes A_{*}$ be the comultiplication map. Then 
for all $k \ge 0$,

\begin{equation}\Delta(\bar{\xi}_{k}) = 1 \otimes \bar{\xi}_{k} + \Sigma_{j=1}^{k-1}\bar{\xi}_{j} \otimes \bar{\xi}_{k-j}^{p^{j}} + \bar{\xi}_{k} \otimes 1\end{equation}

\begin{equation}\Delta(\bar{\tau}_{k}) = 1 \otimes \bar{\tau}_{k} + \Sigma_{j=1}^{k-1}\bar{\tau}_{j} \otimes \bar{\xi}_{k-j}^{p^{j}} + \bar{\tau}_{k} \otimes 1 .\end{equation}

Let $Q_{i}$ denote the $i^{th}$ Milnor primitive, and let $E(n)$ denote the subalgebra $E(Q_{0}, Q_{1}, \ldots, Q_{n})$ of the Steenrod algebra. 

At odd primes, $Q_{i}$ is the dual of $\bar{\tau}_{i}$.

For all $n \ge 0$, we can define an $E(n)_{*}$-coaction \[\psi: A_{*} \rightarrow E(n)_{*} \otimes A_{*}\] to be the composition

\[ 
\begin{tikzcd}
 \psi: A_{*} \arrow[r, "\Delta"]
 &
  A_{*} \otimes A_{*} \arrow[r, two heads]
 &
 E(n)_{*} \otimes A_{*}
\end{tikzcd}
  \]
where the second map is the natural projection. 

Specifically, \begin{equation}\psi(\bar{\xi}_{k}) = 1\otimes \bar{\xi}_{k}\end{equation}

\begin{equation}\psi(\bar{\tau}_{k}) = \sum\limits_{i=0}^{min(k,n)} \bar{\tau}_{i} \otimes \bar{\xi}_{k-i}^{p^{i}} + 1 \otimes \bar{\tau}_{k} .\end{equation}

Recall from \cref{cor:1:equiv} that there is an equivalence of categories $\varphi: Comod_{E(n)_{*}} \rightarrow Mod_{E(n)}$ for all $n \ge 0$. So we can instead consider the $E(n)$-module structures induced by $\psi$, which are sometimes easier to deal with.

\begin{proposition}
    Let 
\[\mu: E(n) \otimes A_{*} \rightarrow A_{*} \]

denote the $E(n)$ action induced by $\psi$ (\cref{lem:1:montgomery1}). 

At odd primes, 
\[ Q_{j}\bar{\tau}_{k} = \begin{cases}  \bar{\xi}_{k-j}^{p^{j}} & \text{ if } j\le k\\
 0 & \text { if } j > k\end{cases} \]

\[Q_{j}\bar{\xi}_{k} = 0 \text{ for all $k \in \N$}.\]

\end{proposition}

\begin{proof}
Recall that for all $j$, $Q_{j}$ is the dual of $\bar{\tau}_{k}$. 

So \begin{equation}
Q_{j}(\bar{\xi}_{k}) = \langle Q_{j},1 \rangle \otimes \bar{\xi}_{k} = 0  
\end{equation}
\begin{equation}
Q_{j}(\bar{\tau}_{k})  = \sum\limits_{i=0}^{min(k,n)} \langle Q_{j}, \bar{\tau}_{i} \rangle \bar{\xi}_{k-i}^{p^{i}} + \langle Q_{j}, 1 \rangle \otimes \bar{\tau}_{k} = \begin{cases}  \bar{\xi}_{k-j}^{p^{j}} & \text{ if } j\le k\\
 0 & \text { if } j > k
\end{cases}.
\end{equation}
\end{proof}

Recall that \[A//E(i)_{*} \cong \F_{p}[\bar{\xi}_{1}, \bar{\xi}_{2} \cdots ] \otimes E(\bar{\tau}_{i+1}, \bar{\tau}_{i+2} \cdots )\] is a sub-comodule algebra of $A_{*}$. So the action of $E(n)$ on $A//E(i)_{*}$ is also given by the formula above. 

\subsubsection{The $2$-primary case}\label{action:even}
At the prime $p=2$, 
\[ A_{*} \cong \F_{2}[\xi_{1}, \xi_{2}, \ldots ], \]

where $|\bar{\xi}_{i}| = 2p^{i} - 1$. As in the odd primary case, we will work with the conjugates $\bar{\xi}_{i}$ of the usual generators. The comultiplication map $\Delta: A_{*} \rightarrow A_{*} \otimes A_{*}$ has the same formula as in the odd-primary case. That is, 

\begin{equation}\Delta(\bar{\xi}_{k}) = 1 \otimes \bar{\xi}_{k} + \Sigma_{j=1}^{k-1}\bar{\xi}_{j} \otimes \bar{\xi}_{k-j}^{p^{j}} + \bar{\xi}_{k} \otimes 1\end{equation}

At the prime $p=2$, $Q_{i}$ is the dual of $\bar{\xi}_{i+1}$. Using the same technique as in the odd case, we arrive at the action of $Q_{i}$ on $A//E(n)_{*}$ at the prime $p=2$. At the prime $2$, 

\[ A//E(n)_{*} \cong \F_{p}[\bar{\xi_{1}}^{2}, \ldots ,\bar{\xi_{n+1}}^{2}, \bar{\xi}_{n+2}, \bar{\xi}_{n+3}, \ldots] .\]

\begin{proposition}
    Let 
\[\mu: E(n) \otimes A_{*} \rightarrow A_{*} \]

denote the $E(n)$ action induced by $\psi$ (\cref{lem:1:montgomery1}). 

At the prime $p=2$, \[Q_{j}\bar{\xi}_{k} =  \begin{cases}  \bar{\xi}_{k-j-1}^{2^{j+1}} & \text{ if } j < k\\
 0 & \text { if } j \ge k\end{cases}.\]
\end{proposition}

\subsubsection{The homology of $BP\n$}
Next we will recall the homology of $BP\n$. We start with the following result of S. Wilson.

\begin{theorem}\cite[Prop.~1.7]{wilson1975omega}
Let $n \ge 0$. Then 
\[ H^{*}BP\n \cong A//E(Q_{0}, Q_{1}, \ldots Q_{n}). \]
\end{theorem}

\begin{lemma}\label{tensorcotensor}
Let $\Sigma$ be a Hopf algebra. Let $\psi_{N}: N \rightarrow \Sigma \otimes N$ be a left $\Sigma$-comodule which is finitely generated in each degree, and let $\psi_{M}: M \rightarrow M \otimes \Sigma$ be a right $\Sigma$-comodule which is also finitely generated in each degree. Then $(M \square_{\Sigma} N)^{*} \cong M^{*} \otimes_{\Sigma^{*}}N^{*}$.
\end{lemma}

\begin{proof}
First, note that $M \square_{\Sigma} N$ is defined to be the equalizer
\[
\begin{tikzcd}[ampersand replacement=\&]
M \square_{\Sigma} N \arrow[r] \arrow[r]
\&
M \otimes N \arrow[r, shift left = .5ex, "\psi_{M} \otimes 1"]\arrow[r, shift right = .5ex, "1 \otimes \psi_{N}" below]
\&
M \otimes \Sigma \otimes N
\end{tikzcd}
.\]

Let $\mu_{M}: \Sigma^{*} \otimes M^{*} \rightarrow M^{*}$ and $\mu_{\Gamma}: \Gamma^{*} \otimes \Sigma^{*} \rightarrow \Sigma^{*}$ be the dual actions induced by $\psi_{M}$ and $\psi_{\Gamma}$.

Taking the dual of the equalizer diagram above, we get a coequalizer
\[
\begin{tikzcd}
(M \square_{\Sigma} N)^{*}
&
M^{*} \otimes N^{*} \arrow[l]
&
M^{*} \otimes \Sigma^{*} \otimes N^{*} \arrow[l, shift right = .5ex, "\mu_{M}" above]\arrow[l, shift left = .5ex, "\mu_{N}" below]
\end{tikzcd}
.\]

This is exactly the coequalizer diagram that defines $M^{*} \otimes_{\Sigma^{*}} N^{*}$.

\end{proof}

\begin{remark}
Let $E(n)$ denote the $A$-subalgebra $E(Q_{0}, Q_{1}, \ldots, Q_{n})$. Recall that $A//E(n)$ is defined to be $A \otimes_{E(n)}\F_{p}$. So \cref{tensorcotensor} tells us that the dual of $A//E(n)$, denoted $A//E(n)_{*}$, is $A//E(n)_{*} \cong A_{*} \square_{E(n)_{*}}\F_{p}$.
\end{remark}

\begin{corollary}\label{hom:BPn}
The homology of $BP\n$ is $H_{*}BP\n \cong A//E(n)_{*}$.
\end{corollary}

\subsection{The homology-level splitting}\label{maps}
Fix a prime $p$, and let $q = 2(p-1)$. Recall that for all $i \ge 0$, $A//E(i)_{*}$ is a sub-comodule algebra of $A_{*}$. We will let $A//E(-1)_{*}$ denote $A_{*}$.

Define a weight function $wt$ on the monomials of $A_{*}$ by setting $wt(\bar{\xi}_{j}) = wt(\bar{\tau}_{j}) = p^{j}$, and let $wt(xy) = wt(x) + wt(y)$.

Let the $j^{th}$ Brown-Gitler comodule $B_{i}(j)$ be the subspace of $A//E(i)_{*}$ generated by monomials of weight at most $j$, and let $M_{i}(j)$ be the subspace of $A//E(i)_{*}$ generated by monomials of weight exactly $j$. These subspaces have the following subcomodule structures.

\begin{lemma}\label{basic 1} \cite[Prop~4.6]{dominic_odd}, \cite[Prop~3.3]{dominic_even}

Let $i \ge -1$. Then for all $k \ge 0$, $M_{i}(k)$ is an $E(i)_{*}$-subcomodule of $A//E(i)_{*}$, and $B_{i}(k)$ is an $E(i+1)_{*}$-subcomodule of $A//E(i)_{*}$.
\end{lemma}

\begin{lemma}\label{lem:2:iso}\cite[Lemma~4.8]{dominic_odd}\cite[Prop~3.3]{dominic_even}
Let $i \ge -1$. Then for all $k\ge 0$, there exists an $E(i+1)$-module isomorphism  
\[\theta_{k}:  \Sigma^{qk}B_{i}(k) \rightarrow M_{i+1}(pk). 
\]
\end{lemma}

Note that $A//E(i)_{*} \cong_{E(i)_{*}} \bigoplus\limits_{k=0}^{\infty} M_{i}(pk)$, so by \cref{lem:2:iso} we have the following isomorphism. 

\begin{theorem}\label{eiso}\cite[Lemma~4.8]{dominic_odd}\cite[Prop~3.3]{dominic_even}
Let $n \ge 0$. Then there exists a family of maps \[\{ \theta_{k}: \Sigma^{qk}B_{n-1}(k) \rightarrow H_{*}BP\n |\ k \in \N \}\] such that their sum 
\[\bigoplus\limits_{k=0}^{\infty} \theta_{k}:\bigoplus\limits_{k=0}^{\infty}\Sigma^{qk}B_{n-1}(k) \rightarrow H_{*}BP\n    \]
is an isomorphism of $E(n)_{*}$-comodules.
\end{theorem}

In \cref{sec:BP2mod}, we will discuss Baker-Lazarev's relative Adams spectral sequence in the category of $BP\2$-modules (\cref{prop:bakerlazarev2}). Then in \cref{sec:ass}, we will explain how this spectral sequence can be used at height $n=2$ to lift \cref{eiso} to the splitting

\[ BP\2 \smash BP\2 \simeq \bigoplus\limits_{k=0}^{\infty} BP\2 \smash \Sigma^{qk}B_{1}(k).\]

\section{The relative Adams spectral sequence}\label{sec:bakerLazarev}

\subsection{The category of $BP\2$-modules}\label{sec:BP2mod}

\subsubsection{Definitions and basic results}

In this section we recall Baker-Lazarev's results on the category of $R$-modules, and in particular their adaptation of the Adams spectral sequence to this setting. We also use Katth\"an-Tilson's construction of the relative K\"unneth spectral sequence. Let $R$ be a homotopy commutative ring spectrum. We start by recalling the category of $R$-module spectra as discussed in EKMM. Let $R$ be an $S$-algebra ($A_{\infty}$ ring spectrum). Let $M$ be a right $R$-module, and let $N$ be a left $R$-module. We let $M \smash_{R} N$ denote the coequalizer \cite[Def~III.3.1]{ekmm}
\[
\begin{tikzcd}
{M  \smash R  \smash N} \arrow[r, shift right] \arrow[r, shift left] & {M \smash N}  \arrow[r] & {M \smash_{R} N}.\\
\end{tikzcd}
\]

Note that by definition, 
\begin{equation} \label{smash:id}
M \simeq M \smash_{BP\2} BP\2 
\end{equation}
\begin{equation} 
N \simeq BP\2 \smash_{BP\2} N
.\end{equation}

This $R$-module smash product commutes with the ordinary smash product, so that if $M$ is a right $R$-module, $N$ is an $R$-bimodule and $L$ is a left $R$-module,

\begin{equation} M \smash_{R} N \smash L \simeq M \smash N \smash_{R} L.\end{equation}\label{smash:commute}

Let $E$ be an $R$-algebra, and let $M$ be an $R$-module spectrum. Then the $E$-homology in the category of $R$-modules is defined to be \[E_{*}^{R}(M) = \pi_{*}(E \smash_{R} M).\] 

It will be helpful to note that 
\begin{equation}\label{compare homology}
    E_{*}^{R}\big(R \smash M) \cong E_{*}M.
\end{equation}

We will use the following special case of the K\"unneth formula for $R$-modules.

\begin{proposition}(K\"unneth formula for $R$-modules)\label{lem:kunneth}
    Let $Y$ be an $R$-module spectrum. Then 
    \[ H_{*}R \otimes H_{*}^{R}Y \cong H_{*}Y .\]
\end{proposition}

\begin{proof}
Let $X, Y$ be any $R$-module spectra. By \cite[Thm~IV.4.1]{ekmm}, there exists a K\"unneth spectral sequence 
\[ E_{2}^{*,*} \cong Tor^{H_{*}}(H_{*}X, H_{*}^{R}Y) \Longrightarrow H_{*}(X \smash_{R} Y) .\]
Note that $H_{*}R$ is flat over $H_{*}$, so the spectral sequence collapses at the $E_{2}$-page, with $E_{2}^{*,*} \cong H_{*}X \otimes H_{*}^{R}Y$. In the case $X = R$, this yields
\[ H_{*}R \otimes H_{*}^{R}Y \cong H_{*}Y.\]
\end{proof}

We will also need to use a version of the Whitehead theorem in the category of $R$-modules. 

\begin{proposition}(Whitehead Theorem for $R$-modules)\label{whitehead}
Let $X,Y$ be $p$-complete spectra which are also $BP\2$-modules. If $\varphi: X \rightarrow Y$ is a $BP\2$-module map such that $H_{*}^{BP\2}\varphi$ is an isomorphism, then $\varphi$ is a homotopy equivalence.
\end{proposition}

\begin{proof}
    Suppose that $X$ is a $p$-complete spectrum. Recall that $p$-complete spectra are $H$-local in the category of $S$-module spectra. We will show that $X$ is also $H$-local in the category of $BP\2$-modules. 
    
    Suppose $A$ is $H$-acyclic in the category of $BP\2$-modules, that is $H_{*}^{BP\2}A = 0$. Recall from \cref{lem:kunneth} that 
    \[ H_{*}A \cong H_{*}BP\2 \otimes H_{*}^{BP\2}A,\]
    so $H_{*}A = 0$. By \cite[Lemma~1.5]{bousfield}, $BP\2 \smash A$ is also $H$-acyclic in the ordinary category of $S$-modules. So
    \[ [A,X]^{BP\2} \simeq [BP\2 \smash A, X] = 0 .\]
    So $X$ is indeed $H$-local in the category of $BP\2$-modules.

So our map $\varphi: X \rightarrow Y$ of $p$-complete modules is in fact a map of $H$-local modules in the $BP\2$-module category. Let $F$ denote the fiber of $\varphi$. We are assuming that $H_{*}^{BP\2}\varphi$ is an isomorphism, so $F$ must be $H$-acylic in the category of $BP\2$-modules. By definition, $[F,X]^{BP\2} = 0$, and so $\varphi: X \rightarrow Y$ is an equivalence in the category of $BP\2$-modules. It follows that $\varphi$ is still an equivalence after forgetting the $BP\2$-module structure.
\end{proof}

\subsubsection{The relative Adams spectral sequence}

Now we will give an explicit description of $H_{*}^{BP\2}H$. In \cite{katthan2017homology}, K\"atthan-Tilson showed that the relative Kunneth spectral sequence constructed in \cite[IV.4.4]{ekmm} is multiplicative under certain conditions.

\begin{theorem}\cite[Thm~2.12]{katthan2017homology}
    Let $R$ be an $\mathbb{E}_{3}$ algebra, and let $A$ and $B$ be commutative $S$-algebras over $R$ which are cofibrant $R$-modules. Then the K\"unneth spectral sequence
    \[ Tor_{s}^{R_{*}}(A_{*}, B_{*}) \Longrightarrow \pi_{*}(A \smash B)\]
    is multiplicative.
\end{theorem}

Baker-Lazarev used the multiplicativity of the K\"unneth spectral sequence to prove the following theorem. The commutativity of the K\"unneth spectral sequence was originally stated as Theorem 2.12 of \cite{bakerlazarev}. Although a flaw was found in the proof of \cite[Thm~2.12]{bakerlazarev}, Tilson addressed this issue and gave an alternative proof for the case where $R$ is an $E_{\infty}$-algebra in \cite{tilson2016power}. Katth\"an and Tilson then generalized this result to $\E_{3}$-commutative algebras in \cite[Thm~2.12]{katthan2017homology}. 

Baker-Lazarev's statement, \cite[Prop~1.2]{bakerlazarev}, is given only for $R$ an $\E_{\infty}$-commutative algebra. However, their proof of \cite[Prop~1.2]{bakerlazarev} relies solely on the multiplicativity of the K\:unneth spectral sequence, so \cite[Thm~2.12]{katthan2017homology} allows us to restate their theorem here for $R$ an $\E_{3}$-commutative algebra. 

\begin{proposition}\label{prop:bakerlazarev1}\cite[Prop~1.2]{bakerlazarev}
Let $R$ be an $\mathbb{E}_{3}$-commutative algebra. Let $E = R/I$ be a regular quotient such that $R/I$ is a commutative $S$-algebra, and $I_{*}$ is generated by the regular sequence $u_{0}, u_{1}, \ldots$. Then as an $E_{*}-algebra$, $E_{*}^{R}E = E(\alpha_{i}| \ i \ge 0)$, where $|\alpha_{i}| = |u_{i}| + 1$. 
\end{proposition}
Recall that the mod-$p$ Eilenberg-Maclane spectrum $H$ can be described as a quotient of $BP\n$ by the regular sequence $(p,v_{1},\ldots, v_{n})$. In \cite{hahn2022redshift}, Hahn and Wilson showed that at each prime $p$ and for each height $n \in \N$, there exists an $\E_{3}$ $BP-$algebra structure on $BP\langle n \rangle $. So we can apply \cref{prop:bakerlazarev1} to describe $H_{*}^{BP\2}H$ as follows. 

\begin{proposition}\label{prop:bakerlazarev:ex:1}
As an $\F_{p}$-algebra, $H_{*}^{BP\n}H \cong E(n)_{*}$.
\end{proposition}

\begin{proof}
By \cref{prop:bakerlazarev1}, $H_{*}^{BP\n}H \cong E(\beta_{j}| \ j > n)$, where $|\beta_{j}| = |v_{j}| + 1$. Recall that $|v_{j}| = 2p^{j} - 1$. 

At odd primes, $E(n)_{*} = E( \bar{\tau}_{0}, \ldots \bar{\tau}_{n} )$, with $|\bar{\tau}_{i}| = 2p^{i}-1 = |v_{i}| + 1$. 

At the prime $p=2$, $E(n)_{*} = E(\bar{\xi}_{1}, \ldots, \bar{\xi}_{i+1})$ with $ |\bar{\xi}_{j} | = 2p^{j}-1 = |v_{j}| + 1$. 

Note that each prime, the dimensions of the generators of $E(n)$ and $H_{*}^{BP\2n}H$ match. So indeed $H_{*}^{BP\n}H \cong E(n)_{*}$.
\end{proof}

We will use Baker-Lazarev's adaptation of the Adams spectral sequence, the relative Adams spectral sequence in the category of $R$-modules. (In their paper, Baker and Lazarev work in the category of commutative $S$-algebras. However, their construction of the relative Adams spectral sequence only uses homotopy commutativity).

\begin{proposition}\label{prop:bakerlazarev2}\cite[Prop~2.1]{bakerlazarev}
    Let $L,M$ be $R$-modules, and let $E$ be a homotopy commutative ring spectrum with $E_{*}^{R}E$ flat as a left or right $E_{*}$-module. If $E_{*}^{R}L$ is projective as an $E_{*}$-module, then there is an Adams spectral sequence with \[ E_{2}^{s,t} = Ext_{E_{*}^{R}E}^{s,t}(E_{*}^{R}L, E_{*}^{R}M).\] 
\end{proposition}

So for the case $R = BP\2$ and $E = H$, we have the following Adams spectral sequence:
\[ E_{2}^{s,t} = Ext_{E(2)_{*}}^{s,t}(H_{*}^{BP\2}L, H_{*}^{BP\2}M)\]
for any $BP\2$-modules $L$ and $M$. This spectral sequence weakly converges to $[L, \widehat{M_{E}^{R}}]^{BP\2}$, where $\widehat{M_{E}^{R}}]^{BP\2}$ denotes the $E$-nilpotent completion of $M$ in the category of $R$-modules. 

Let $E^{\bullet}$ be the cosimplicial spectrum with $E^{k} = E^{\smash k}$. The coface maps are induced by the unit maps $S^{0} \rightarrow E$, and the degeneracies by the multiplication $E \smash E \rightarrow E$. The ordinary $E$-nilpotent completion of $M$ can be defined as the totalization of $E ^{\bullet} \smash M$. The $E$-nilpotent completion of $M$ in the category of $R$-modules is defined analogously.
\begin{definition}
    Let $E^{\bullet_{R}}$ be the cosimplicial spectrum with $E^{k} = E^{\smash_{R} k}$. The coface maps are induced by the unit maps $R \rightarrow E$, and the degeneracies by the multiplication $E \smash_{R} E \rightarrow E$. The $E$-nilpotent completion of $M$ in the category of $R$-modules is \[X_{E}^{R} = Tot\big(E ^{\bullet_{R}} \smash X\big).\]  

\end{definition}

(If we take $R$ to be the sphere spectrum $S$, we recover the ordinary $E$-nilpotent completion of $M$).

Let $A \overset{\alpha}{\rightarrow} B \rightarrow E$ be a diagram of $(-1)$-connected homotopy commutative $S$-algebras, and let $M$ be a $B$-module spectrum. Let $\widehat{M_{E}^{B}}$ denote the $E$-nilpotent completion of $M$ in the category of $B$-modules. Note that $\alpha: A \rightarrow B$ induces a natural $A$-module structure on $M$, so we let $\widehat{M_{E}^{A}}$ denote the $E$-nilpotent completion of $M$ in the category of $A$-modules. Note that for any pair of $B$-modules $M,N$, there is a natural homomorphism $M \smash_{A} N \rightarrow M \smash_{B} N$ induced by $\alpha$. So we can consider the induced homomorphism $\widehat{M_{E}^{A}} \rightarrow \widehat{M_{E}^{B}} $, and compare the two $E$-nilpotent completions.   

\begin{remark}
    In Carlsson's paper \cite{carlsson2008derived}, he works in the category of $\E_{\infty}$-commutative $S$-algebras. However, his proof of \cite[Thm~6.10]{carlsson2008derived} only relies on homotopy commutativity, so we restate the result here for homotopy commutative spectra. 
\end{remark}

\begin{theorem}\cite[Thm~6.10]{carlsson2008derived}
Let $A \overset{\alpha}{\rightarrow} B \rightarrow E$ be a diagram of $(-1)$-connected homotopy commutative $S$-algebras, and let $M$ be a $B$-module spectrum. Suppose that the homomorphism $\pi_{0}\alpha: \pi_{0}A \rightarrow \pi_{0}B$ is an isomorphism. Then the natural homomorphism $\widehat{M_{E}^{A}} \rightarrow \widehat{M_{E}^{B}}$ is a weak equivalence of spectra.
\end{theorem}

It follows that the $H$-nilpotent completion of a $BP\2$-module $M$  $\widehat{M_{H}^{BP\2}}$ is weakly equivalent to the ordinary $H$-nilpotent completion $\widehat{M_{H}^{S}}$. Since we are working up to $p$-completion throughout this paper, we will simply say that the spectral sequence 
\[ E_{2}^{s,t} = Ext_{E(2)_{*}}^{s,t}(H_{*}^{BP\2}L, H_{*}^{BP\2}M)\]
weakly converges to $BP\2$-module maps from $L$ to $M$, denoted $[L,M]^{BP\2}$. 

We wish to consider the case $L = BP\2 \smash l_{k}$, $M = BP\2 \smash BP\2$. Recall that for any $BP\2$-module $X$, \[H_{*}^{BP\2}BP\2 \smash X \cong H_{*}X,\]

so we can write the $E_{2}$-page of this spectral sequence as

\[ E_{2}^{s,t} \cong Ext_{E(2)_{*}}(H_{*}l_{k}, H_{*}BP\2).\]

Furthermore, note that $H_{*}l_{k}$ is a finite module, while $H_{*}BP\2$ is connective finitely generated in each degree. So by Boardman's convergence criterion \cite[Thm~7.1]{boardman1999conditionally}, we have the following strongly convergent spectral sequence.

\begin{proposition}
Let $n \ge 0$. The for all $k \in \N$, there exists a strongly convergent Adams spectral sequence
\[ E_{2}^{s,t} = Ext_{E(2)_{*}}(H_{*}l_{k}, H_{*}BP\2) \Longrightarrow [BP\2 \smash l_{k}, BP\2 \smash BP\2]. \]
\end{proposition}

\subsection{How we will use the relative Adams spectral sequence}\label{sec:ass}

Recall from \cref{eiso} that there exists a family of $E(2)_{*}$-comodule maps \[ \bigoplus\limits_{k=0}^{\infty} \theta_{k}: \Sigma^{qk}H_{*}l_{k} \longrightarrow H_{*}BP\2 \]
such that $\bigoplus\limits_{k=0}^{\infty}\theta_{k}$ is an isomorphism. 

By \cref{compare homology}, we can think of this as a family of $E(2)_{*}$-comodule maps
\[ \bigoplus\limits_{k=0}^{\infty} \theta_{k}: \Sigma^{qk}H_{*}^{BP\2}\big( BP\2 \smash l_{k} \big) \longrightarrow H_{*}^{BP\2}\big(BP\2 \smash BP\2 \big).\] 

We will use the following relative Adams spectral sequence to lift these maps to the level of spectra.

\begin{proposition}\label{relASS}
Let $n \ge 0$. Then for all $k \in \N$, there exists a strongly convergent Adams spectral sequence
\[ E_{2}^{s,t} = Ext_{E(n)_{*}}(H_{*}B_{n}(k), H_{*}BP\n) \Longrightarrow [BP\n \smash B_{n-1}(k), BP\n \smash BP\n]^{BP\n}. \]
\end{proposition}

\begin{proof}

By \cref{prop:bakerlazarev2}, there exists an Adams spectral sequence

\[ E_{2}^{s,t} = Ext_{E(2)_{*}}\big(\Sigma^{qk}H_{*}l_{k}, H_{*} BP\2 \big) \Longrightarrow [BP\2 \smash l_{k}, BP\2 \smash BP\2 ]^{BP\2} .\]

Note that $H_{*}{l}_{k}$ is finite, while $H_{*}BP\2 \smash BP\2$ is connective and finitely generated in each degree. It follows that $Ext_{A_{*}}^{s,t}(H_{*}l_{k}, H_{*}BP\2^{\otimes 2})$ is finite for each pair $(s,t)$. So by \cite[Thm~7.1]{boardman1999conditionally}, the spectral sequence strongly converges to $[BP\2 \smash l_{k}, BP\2 \smash BP\2 ]^{BP\2}$. 
\end{proof}

We can think of $\theta_{k}$ as a class in degree $E_{2}^{0,0}$. Suppose that we can show that each map $\theta_{k}$ survives its spectral sequence. Then we will have constructed a family of $BP\2$-module maps 

\[\widetilde{\theta_{k}}: BP\2 \smash l_{k} \rightarrow BP\2 \smash BP\2\] 

whose sum induces the isomorphism $\bigoplus_{k=0}^{\infty}\theta_{k}$ on $H_{*}^{BP\2}$-homology. Then it will follow immediately from \cref{whitehead} that \[\bigvee\widetilde{\varphi_{k}}: \bigvee\limits_{k=0}^{\infty} BP\2 \smash l_{k} \rightarrow BP\2 \smash BP\2 \]
is a homotopy equivalence up to $p$-completion. So we will spend the rest of the paper demonstrating that $\theta_{k}$ survives the spectral sequence. 

Note that ${\theta_{k}}$ is in degree$(s,t) = 0$, while the differential has degree $|d_{r}| = (r, r + 1)$. So the potential obstructions to lifting $\varphi_{k}$ are those classes found in $Ext_{E(2)_{*}}^{s,s+1}\big(\Sigma^{qk}H_{*}{l}_{k}, H_{*}BP\2 \big)$, where $s \ge 2$. We will show that all of these potential obstructions survive the spectral sequence.  

It is interesting to note that the analogous situation for $BP\1$, there are no potential obstructions to lifting these maps. So one can immediately recover Kane's splitting of $BP\1 \smash BP\1$ (\cref{eqn5}) using this method. However, there are many potential obstructions to lifting our family of maps 

\[\{ \theta_{k}: \Sigma^{qk}H_{*}{l}_{k} \rightarrow H_{*}BP\2\}.\] 

To deal with these obstructions, we will adapt the method used by Klippenstein to construct a decomposition of $bu \smash BP\n$ in terms of $bu$-modules generated by integral Brown-Gitler spectra. Specifically, he constructed the following splitting.

\begin{theorem} \cite[Thm~11]{klippenstein_splitting} \label{klippenstein}
Let $n \ge 0$, and let $p$ be an odd prime. Then there exists a splitting \[bu \smash BP\n \vee HV \simeq \bigvee\limits_{I \in \N^{n}} \Sigma^{d(I)} bu \smash H\Z_{\lfloor i_{n} \rfloor} \vee HV' \]
where $I = (i_{1}, \ldots, i_{n})$, $d(I) = |\bar{\xi}_{1}^{i_{1}}\bar{\xi}_{2}^{i_{2}}\cdots \bar{\xi}_{n}^{i_{n}} | $, and $HV, HV'$ are Eilenberg-Maclane spectra of $\F_{p}$-vector spaces.

\end{theorem}

\begin{subsubsection}{Why we chose to use a relative Adams spectral sequence}
To construct this splitting, Klippenstein used an Adams spectral sequence 
\begin{equation}\label{klipp:ass} E_{2}^{s,t} = Ext_{A_{*}}\big(\Sigma^{d(I)}H_{*}H\Z_{i_{n}}, bu \smash BP\n  \big) \Rightarrow [H\Z_{i_{n}}, bu \smash BP\n ]\end{equation} 

to construct maps \[\varphi_{I}: H\Z_{k} \rightarrow bu \smash BP\n \]

and then showed that the composite \[bu \smash BP\n \vee HV \simeq \bigvee\limits_{I \in \N^{n}} \Sigma^{d(I)} bu \smash H\Z_{\lfloor i_{n} \rfloor} \vee HV' \]
induces an isomorphism in homology. Note that after a change-of-rings, \[E_{2}^{s,t} = Ext_{A_{*}}\big(\Sigma^{d(I)}H_{*}H\Z_{i_{n}}, bu \smash BP\n  \big) \cong Ext_{E(1)_{*}}\big(\Sigma^{d(I)}H_{*}H\Z_{i_{n}}, H_{*}BP\n  \big). \]

We could have taken an analogous approach for constructing our splitting. Specifically, the $E(2)_{*}$-comodule homomorphism $\theta_{k}: H_{*}\Sigma^{qk}l_{k} \rightarrow H_{*}BP\2$ extends to an $A_{*}$-comodule homomorphism
\[ \overline{\theta_{k}}: H_{*}\Sigma^{qk}l_{k} \rightarrow H_{*}BP\2 \otimes H_{*}BP\2.\] 
It can then be shown algebraically that the composite
\[ 
\begin{tikzcd}
\bigoplus\limits_{k=0}^{\infty} H_{*}BP\2 \otimes H_{*}\Sigma^{qk}l_{k}
\arrow[r, "1 \otimes \overline{\theta_{k}}" ]  
&
H_{*}BP\2^{\otimes 3} \arrow[r, "\mu \otimes 1"]
&
H_{*}BP\2^{\otimes 2}
\end{tikzcd}
\]

induces an isomorphism. We found it more efficient to use a relative Adams spectral sequence instead for several reasons. First, our approach makes it more straightforward to employ the $BP\2$-module structure of $BP\2 \smash l_{k}$ when analyzing the potential obstructions. Furthermore, using a relative Adams spectral sequence saves us the trouble of confirming that the composite above really does induce an isomorphism. Additionally, we found it necessary to use other relative Adams spectral sequences to prove \cref{thm:bp split} at the primes $2$ and $3$, so we needed to include most of the discussion and background in \cref{sec:BP2mod} regardless of which approach we chose.     
\end{subsubsection}

To construct the splitting of \cref{klippenstein}, Klippenstein used an Adams spectral sequence
\begin{equation} E_{2}^{s,t} = Ext_{E(1)_{*}}\big(\Sigma^{d(I)}H_{*}H\Z_{i_{n}}, H_{*} BP\n  \big) \Rightarrow [H\Z_{i_{n}}, bu \smash BP\n ]\end{equation} 

to lift an analogous class of maps \[\{ \theta_{I}: \Sigma^{d(I)} H_{*}H\Z_{i_{n}} \rightarrow bu \smash BP\n | I = (i_{1}, \ldots, i_{n}) \in \N^{n} \}\] to the level of spectra. Just as in our situation, there are many potential obstructions to lifting the maps $\theta_{I}$. To analyze these potential obstructions, Klippenstein used an auxiliary spectral sequence called the universal coefficient spectral sequence.

\section{Auxiliary spectral sequences}\label{spec:aux}
\subsection{The universal coefficient spectral sequence}\label{sec:ucss}
We will start by recalling Robinson's universal coefficient spectral sequence. 
\begin{theorem}\cite[p.~9]{robinson}
Let $R$ be a ring spectrum, and let $G$ be an $R$-module spectrum. Then for any connective ring spectrum $X$, there exists a universal coefficient spectral sequence 
\[ Ext_{R_{*}}^{*,*}\big( R_{*}X, G_{*} \big) \Longrightarrow [X, G] .\]
\end{theorem}

Note that $bu \smash BP\n$ is naturally a $bu$-module, so Klippenstein was able to use a universal coefficient spectral sequence of the following form:

\begin{equation}\label{klipp:ucss}E_{2}^{u,v} = Ext_{bu_{*}}^{u,v}\big(\Sigma^{d(I)}bu_{*}H\Z_{k }, bu_{*}BP\n  \big) \Rightarrow [BP\2 \smash H\Z_{k}, bu \smash BP\n ]_{v-u}^{BP\2}.\end{equation}

Recall that \[[H\Z_{k}, bu \smash BP\n ] \cong [BP\2 \smash H\Z_{k}, bu \smash BP\n ]^{BP\2},\]  so the spectral sequence (\ref{klipp:ucss}) converges to the same target as (\ref{klipp:ass}). Although they converge to the same target, these two spectral sequences have very different filtrations. The $E_{2}$-page of the Adams spectral sequence (\ref{klipp:ass}) has nonzero classes in arbitrarily high filtration $s$. However, the ring $bu_{*}$ has global dimension $2$, which tells us that $Ext_{bu_{*}}^{u,*} = 0$ for all $u \ge 3$. Note that the differential $d_{r}^{UCSS}$ increases the homological degree $u$ by $r$. This implies that every class on the line $Ext_{bu_{*}}^{1,*}$ survives the spectral sequence.

In turn, this tells us that for each nonzero class on the $(u=1)$-line of the universal coefficient spectral sequence, there must be a corresponding nonzero class in $Ext_{A_{*}}^{*,*}\big(\Sigma^{d(I)}H_{*}H\Z_{i_{n}}, H_{*}bu \otimes H_{*}BP\n  \big) $ which survives the Adams spectral sequence. By directly computing both $E_{2}$-pages, Klippenstein showed that the the the potential obstructions to lifting $\varphi_{I}$ must be among the corresponding classes that survive the Adams spectral sequence.

We will adapt this technique, comparing the universal coefficient spectral sequence

\begin{equation}\label{ucss}
E_{2}^{u,v} = Ext_{BP\2_{*}}\big( BP\2_{*}l_{k}, BP\2_{*}BP\2 \big) \Longrightarrow [BP\2 \smash l_{k}, BP\2 \smash BP\2 ]^{BP\2} 
\end{equation}

to the Adams spectral sequence (\ref{main:relASS}) to show that the potential obstructions in (\ref{main:relASS}) survive the spectral sequence. However, our situation is immediately more complicated than the $bu$-analogue for two reasons. First, $BP\2_{*}$ has global dimension $3$ rather than $2$, leaving more room for non-zero differentials in the universal coefficient spectral sequence. Secondly, both $Ext_{BP\2_{*}}^{*,*}$ and $Ext_{E(2)_{*}}^{*,*}$ are far more complicated to compute than their height one analogues. So instead of computing these $Ext$ terms and directly comparing them, we will use another spectral sequence called the \textit{hypercohomology spectral sequence} as a bridge between them. 

\subsection{The hypercohomology spectral sequence}\label{sec:hss}
We will start by recalling the hypercohomology spectral sequence, a construction of Cartan-Eilenberg \cite{cartan_eilenberg}. Let $A$ be a complex of modules over a ring $\Lambda_{1}$, and let $C$ be a complex of modules over a ring $\Lambda_{2}$. Let $T(-,-): Mod_{\Lambda_{1}}^{op} \times Mod_{\Lambda_{2}} \rightarrow Ab$ be an additive functor such that $T(A,-)$ is covariant, $T(-,C)$ is contravariant, and $T(A,C)$ is a module over a ring $\Lambda$. Let $X$ be a projective resolution of $A$ over $\Lambda_{1}$, and $Y$ an injective resolution of $C$ over $\Lambda_{2}$, where  $X^{p, q}$ denotes the $q^{th}$ stage of the projective resolution of $A^{p}$ and $Y_{p, q}$ denotes the $q^{th}$ stage of the injective resolution of $C^{p}$. Then we can consider the double complex $T^{\bullet,\bullet}(X,Y)$ where
\[T^{p,q}(X,Y) = \bigoplus\limits_{p_{1} + p_{2} = p, q_{1} + q_{2} = q}T(X^{p_{1}, q_{2}}, Y_{p_{1}, q_{2}}).\] The double complex has two filtrations: the first filtration 

\[F_{I}^{p} := \bigoplus_{r \ge p} \bigoplus_{q} T^{r,q}(X,Y) \]

and the second filtration 

\[F_{II}^{q} := \bigoplus_{p} \bigoplus_{r \ge q} T^{p,r}(X,Y) .\]

In certain cases, spectral sequences can be associated to these filtrations. 

\begin{definition}\cite[p.324]{cartan_eilenberg}\label{def:regular}
A filtration $F$ of a complex $D$ is said to be regular if there exists a function $u(n)$ such that $H^{n}F^{p}(D) = 0$ for all $p > u(n)$.
\end{definition}

By definition, $X^{p,q}$ and $Y_{p,q}$ are both zero whenever $q<0$. So $H^{p+q}T^{p,*}(X,Y) = 0$ when $p>n$. Therefore, the first filtration is regular with respect to the total cohomology of the complex \cite[p.369]{cartan_eilenberg}.\\

If the second filtration is also regular, then each filtration induces an associated spectral sequence converging to $H^{*}\R T(A,C)$. Specifically, these are the spectral sequences associated to the double complex $T(X,Y)$. We will use the spectral sequence associated to the second filtration.

\begin{proposition}\cite[p.370]{cartan_eilenberg}\label{prop:c_e}

Let $A,C$ be complexes as described above. If the second filtration $F_{II}^{q}$ is regular, then the following spectral sequence is convergent: \[\bigoplus_{p_{1} - p_{2} = p}H^{q}\R T \big(H^{p_{1}}(A), H^{p_{2}}(C) \big) \overset{q}{\Longrightarrow} H^{p+q}\R T(A,C) .\]

\end{proposition}

Now we will apply this spectral sequence to our situation. Let $E(m)$ be the exterior algebra $E(m) = E(Q_{0}, Q_{1}, \ldots, Q_{m} )$, and let $P(m)$ be the polynomial algebra $P(m)=  \F_{p}[v_{0}, v_{1}, \ldots v_{m}]$. Then the right derived functor $\R Hom_{E(m)}(\F_{p}, -)$ can be represented by the chain complex\\ $Hom_{E(m)}(Q^{\bullet}, -)$, where $Q^{\bullet}$ is a projective resolution of $\F_{p}$ over $E(m)$. 

This chain complex is a module over the ring $\R Hom_{E(m)}(\F_{p}, \F_{p})$. Note that \[P(m) \cong \R Hom_{E(m)}(\F_{p}, \F_{p}).\] Let \[G: \D(E(m)) \rightarrow \D(P(m))\] denote the derived functor $G(-) = \R Hom_{E(m)}(\F_{p}, -)$.

We will need to use the following property of $P(m)$.

\begin{proposition}\cite[Cor~19.7]{eisenbud2013commutative}\label{prop:dim}
Let $L$ be an ordinary module over $P(m)$. Then the global dimension of $L$ is at most $m$.
\end{proposition} 

We also need to note that the functor $G: \D(E(m)) \rightarrow \D(P(m))$ preserves the internal grading $t$.

\begin{proposition}\label{prop:internal_degree}
The functor $G: \D(E(m)) \rightarrow \D(P(m))$ preserves the internal degree $t$.
\end{proposition}

\begin{proof}
Consider the functor $G: \D(E(m)) \rightarrow \D(P(m))$. Let $f: M \rightarrow N$ be an $E(m)$-module map. Let $Q^{\bullet}$ be a projective resolution of $\F_{p}$ over $E(m)$. Then $G(f): G(M) \rightarrow G(N)$ can be represented by the chain complex homomorphism pictured below.

\[ 
\begin{tikzcd}[ampersand replacement=\&]
\cdots \arrow[r]
\&
Hom_{E(m)}(Q^{r}, M) \arrow[r] \arrow[d, "f_{*}"]
\&
Hom_{E(m)}(Q^{r+1}, M) \arrow[r] \arrow[d, "f_{*}"]
\&
\cdots\\
\cdots \arrow[r]
\&
Hom_{E(m)}(Q^{r}, N) \arrow[r]
\&
Hom_{E(m)}(Q^{r+1}, N) \arrow[r]
\&
\cdots\\
\end{tikzcd}
\]

Each $f_{*}: Hom_{E(m)}(P^{r}, M) \rightarrow Hom_{E(m)}(P^{r+1}, M)$ has degree $|f|$, so $G(f)$ has degree $|f|$. 

\end{proof}

Now we can state a special case of \cref{prop:c_e}

\begin{lemma}\label{lem:hss}
Let $M, N$ be $E(2)_{*}$-comodules. Let $G(-)$ denote the right derived functor $\R Hom_{E(2)_{*}}(\F_{p}, -)$,\\ and let $\P(2)$ denote $G(\F_{p})$. Then there exists a spectral sequence 
\[ \bigoplus \limits_{ \substack{r= r_{2} - r_{1} \\ t = t_{2} - t_{1} } }Ext_{\P(2)}^{u}( \big( Ext_{E(2)_{*}}^{r_{1}, t_{1}}(\F_{p}, M), Ext_{E_(2)_{*}}^{r_{2}, t_{2}}(\F_{p}, N) \big) \Rightarrow H^{ r + u, t}\R Hom_{\P(2)}\big( G(M), G(N) ) \big).\]
\end{lemma}

\begin{proof}
Recall that $M$ and $N$ have an induced $E(2)$-module structure as described in \cref{cor:1:equiv}. Let $X$ be a projective resolution of the complex $G(M)$, and let $Y$ be an injective resolution of $G(N)$. By \cref{prop:dim}, the resolutions $X$ and $Y$ can be chosen so that $X^{*,q}$ and $Y^{*,q}$ are both zero for all $q>m$. So $H^{n}F_{II}^{*,q} = 0$ for all $q>m$, and the filtration is regular. So we have the following spectral sequence described in \cref{prop:c_e}:

\[\bigoplus_{p_{1} - p_{2} = p}H^{q}\R Hom_{P(m)} \big(H^{p_{1}}G(M), H^{p_{2}}G(N) \big) {\Longrightarrow} H^{p+q}\R Hom_{P(m)}(G(M),G(N)) .\]

Finally, we can apply \cref{cor:1:equiv} to exchange $Ext_{E(m)}$ for $Ext_{E(m)_{*}}$. So we can express the hypercohomology spectral sequence in the form in which we will use it:

\begin{equation}\label{kmhss}
\begin{split}\bigoplus \limits_{ \substack{r= r_{2} - r_{1} \\ t = t_{2} - t_{1} } }Ext_{\P(2)}^{q}( \big( Ext_{E(2)_{*}}^{r_{1}, t_{1}}(\F_{p}, H_{*}l_{k}), Ext_{E(2)_{*}}^{r_{2}, t_{2}}(\F_{p}, H_{*}BP\2) \big)\\
\Longrightarrow H^{ r + q, t}\R Hom_{\P(2)}\big( G(H_{*}l_{k}), G(H_{*}BP\2 ) \big).
\end{split}
\end{equation}\end{proof}

Now we will use Koszul duality to show that we can rewrite $H^{s,t}\R Hom_{P(m)}\big(G(M), G(N) \big)$ as $Ext_{E(m)}^{s,t}(M,N)$ under certain conditions. First we will recall some definitions.

\begin{definition}
Let $R$ be a differential graded algebra, and let $\D(R)$ denote the derived category of differential graded modules over $R$. A derived module $C$ over $R$ is said to be coherent if $H_{*}C$ is finite. Let $\D(R)^{f}$ denote the full subcategory of $\D(R)$ generated by coherent derived $E$-modules.  
\end{definition}

\begin{definition}\label{def:thick}\cite[p.9]{avramov2010homology}
A full subcategory $\mathcal{C}$ of $\D(R)$ is said to be thick if it satisfies the following properties.
	\begin{enumerate}
		\item $\mathcal{C}$ is additive.
		\item $\mathcal{C}$ is closed under direct summands.
		\item If 
		\[ C' \rightarrow C \rightarrow C''  \] is an exact triangle in $\D(R)$ such that two of the three modules are contained in $\mathcal{C}$, then so is the third. 
	\end{enumerate}
\end{definition}

\begin{definition}\cite[2.1.4]{avramov2010homology}
If $C$ is a derived module over $\D(R)$, then $Thick_{R}(C)$ denotes the intersection of all thick subcategories containing the derived module $C$. 
\end{definition}

Koszul duality can be used to compare modules over $E(m)$ to modules over $P(m)$. Specifically, we will use the following two results.

\begin{lemma}\cite[ Remark~7.5]{avramov2010homology}\label{b:equiv}
Let $R$ be a derived algebra over a field $k$. Then the subcategories $Thick_{R}(k)$ and $\D(R)^{f}$ are equivalent.
\end{lemma}

\begin{theorem}\cite[Remark 7.4]{avramov2010homology}\cite[Thm~4.1]{carlson_iyengar}\label{a:thm:kosz:equiv}
The assignment \[G: \D\big(E(m)\big) \rightarrow \D\big(P(m)\big)\]
\[ C \mapsto  \R Hom_{E(m)} (\F_{p}, C) \]
is an exact functor. Furthermore, $G$ restricts to an equivalence

\[ G: Thick_{E(m)}\big(k \big) \rightarrow Thick_{P(m)}\big(P(m)\big) .\]

\end{theorem}

We can use \cref{b:equiv} to restate \cref{a:thm:kosz:equiv} as follows.

\begin{theorem}\label{thm:kosz:equiv}
The assignment \[G: \D\big(E(m)\big) \rightarrow \D\big(P(m)\big)\]
\[ C \mapsto  \R Hom_{E(m)} (\F_{p}, C) \]
is an exact functor. 

Furthermore, $G$ restricts to an equivalence

\[ G: \mathcal{D}\big(E(m)\big)^{f} \rightarrow Thick_{P(m)}\big(P(m)\big) .\]
\end{theorem}

Recall that although $H_{*}BP\2$ is not a coherent module, it is an infinite sum of coherent modules. We will use the fact that this sum is finite in each topological degree $t$ to show the following corollary to \cref{thm:kosz:equiv}. 

\mhss

\begin{proof}
Recall from \cref{eiso} that $H_{*}BP\2$ splits as a sum of finite $E(2)$-modules, that is,

\[ H_{*}BP\2 \cong_{E(2)} \bigoplus\limits_{m=0}^{\infty}\Sigma^{qk} H_{*}l_{m} .\]

Note that $H_{t}BP\2$ is finite in each degree $t$, so  $G(H_{*}BP\2) \cong G\big(\bigoplus\limits_{k=0}^{\infty}H_{*}l_{k}\big) \cong \bigoplus\limits_{k=0}^{\infty}G(H_{*}\Sigma^{qk}l_{k})$.

It follows that
\[ \R Hom^{u,t}\big(G(H_{*}l_{k}), G(H_{*}BP\2)  \big) \cong \R Hom^{u,t}\big(G(H_{*}l_{k}), \bigoplus\limits_{n=0}^{\infty}\Sigma^{qn}G(H_{*}l_{n})  \big).\]

Likewise, note that $G^{s}(H_{*}l_{k})$ is connective and finite for each degree $s$. It follows that 

\[ \R Hom^{u,t}\big(G(H_{*}l_{k}), \bigoplus\limits_{n=0}^{\infty}\Sigma^{qn}G(H_{*}l_{n})  \big) \cong \bigoplus\limits_{n=0}^{\infty}\Sigma^{qn} \R Hom^{u,t}\big(G(H_{*}l_{k}), G(H_{*}l_{n})  \big).\]

By \cref{thm:kosz:equiv},
\[ \R Hom^{s,t}\big(G(H_{*}l_{k}), G(H_{*}l_{n}) \big) \cong Ext_{E(2)}^{s,t}(H_{*}l_{k}, H_{*}l_{n}).\]

So by \cref{eiso}, \[ \R Hom^{s,t}\big(G(H_{*}l_{k}), G(H_{*}BP\2) \big) \cong Ext_{E(2)_{*}}^{s,t}\big(H_{*}l_{k}, H_{*}BP\2 \big).\]

\end{proof}

So indeed the hypercohomology spectral sequence (\ref{kmhss}) converges to the $E_{2}$-page of the Adams spectral sequence (\ref{main:relASS}).

\section{Constructing a square of spectral sequences}\label{square}

\subsection{Culver's splitting}\label{sec:dominic:splitting}

In this section, we start by recalling Culver's spliting of $BP \smash BP\2$, as well as the analogous splittings of $BP\2 \smash l_{k}$. Then we will use Baker-Lazarev's relative Adams spectral sequence in the category of $BP\2$-modules to show that these are in fact $BP\2$-module splittings.

We start by recalling the following theorem of Margolis.

\begin{theorem}\cite[Thm~2]{margolis_article}\label{thm:marg:4}
Let $Y$ be a bounded below, locally finite spectrum. Then there exists a pair of spectra unique up to equivalence $C$, $V$ such that $V$ is the Eilenberg-Maclane spectrum of an $\F_{p}$-vector space, $H^{*}C$ contains no free $A$-summands, and 

\[ Y \simeq C \vee V .\]
\end{theorem}

Culver combined \cref{thm:marg:4} above with an analysis of the Adams spectral sequence 
\[Ext_{A_{*}}(\F_{p}, H_{*}(BP\2 \smash BP\2) \Longrightarrow \pi_{*}BP\2 \smash BP\2 \]
to show the following result. 

\begin{theorem}\cite[Cor~3.22]{dominic_odd}\label{culver}
At each prime $p$, there exists a splitting unique up to equivalence
\begin{equation}
BP\2 \smash BP\2 \simeq C \vee V
\end{equation}
where 
\begin{enumerate}
    \item $V$ is the Eilenberg-Maclane spectrum of a mod-$p$ vector space, 
    \item and $Ext_{A_{*}}(\F_{p},H_{*}C)$ is $v_{2}$-torsion free.
\end{enumerate}    
\end{theorem}

Substituting $BP\2 \smash l_{k}$ for $BP\2 \smash BP\2$ in the proof of Corollary 3.22 of \cite{dominic_odd} yields an analogous splitting for $BP\2 \smash l_{k}$, as stated in the following theorem.

\begin{theorem}\label{thm:dominic:splitting2}
If $X = BP\2$ or $X = l_{k}$, then then there exists a splitting unique up to equivalence

\[ BP\2 \smash X \simeq C_{X} \vee V_{X}\]

such that $V_{X}$ is a sum (of finite type) of suspensions of $\F_{p}$-Eilenberg-Maclane spectra, and $Ext_{A_{*}}(\F_{p}, H_{*}C_{X})$ is $v_{2}$-torsion free and concentrated in even $(t-s)$-degrees.
\end{theorem}

Now we will use the relative Adams spectral sequence of \cref{prop:bakerlazarev2} to show that these are in fact $BP\2$-module splittings.

\begin{proposition}\label{prop:free}
The $E(2)_{*}$-comodule $H_{*}^{BP\2}V_{X}$ is free and of finite type.  
\end{proposition}

\begin{proof}
The summand $V_{X}$ is a sum of mod-$p$ Eilenberg-MacClane spectra (of finite type), so $H_{*}V_{X}$ is a free $A_{*}$-comodule of finite type. By \cref{lem:kunneth}, 
\[A//E(2)_{*} \otimes H_{*}^{BP\2}V_{X} \cong H_{*}V_{X}.\]

It follows that $H_{*}^{BP\2}V_{X}$ is a free $E(2)_{*}$-comodule of finite type.
\end{proof}

\begin{proposition}\label{thm:bp split}
The splitting 
\[ BP\2 \smash X \simeq C_{X} \vee V_{X}\]
is $BP\2$-module splitting. We will denote the summands $C_{BP\2}$ (resp. $V_{BP\2}$) by $C$(resp. $V$), and we will denote $C_{l_{k}}$ (resp. $V_{l_{k}}$) by $C_{k}$ (resp. $V_{k}$). 
\end{proposition}

\begin{proof}
By \cref{prop:free}, $H_{*}^{BP\2}V_{X}$ is a free summand of $H_{*}V_{X}$, with natural injection and projection maps \[ i: H_{*}^{BP\2}V_{X} \rightarrow H_{*}X \]

and 

\[j:  H_{*}X \rightarrow H_{*}^{BP\2}V_{X} .\]

We can use Baker-Lazarev's adaptation of the Adams spectral sequence (\cref{prop:bakerlazarev2}) to lift these maps to the level of spectra. By \cref{prop:bakerlazarev2}, there exist conditionally convergent spectral sequences

\[ Ext_{E(2)_{*}}^{s,t}(H_{*}^{BP\2}V_{X}, H_{*}X) \Longrightarrow [V_{X}, BP\2 \smash X]^{BP\2}\]
\[ Ext_{E(2)_{*}}^{s,t}(H_{*}X, H_{*}^{BP\2}V_{X}) \Longrightarrow [BP\2 \smash X, V_{X}]^{BP\2}.\]

Note that $H_{*}^{BP\2}V_{X}$ is a free $E(2)_{*}$-comodule, so by \cref{free:proj:inj}, $H_{*}^{BP\2}V_{X}$ is also injective. So both spectral sequences are concentrated on the $s=0$ line. By \cite[Thm~7.1]{boardman1999conditionally}, this implies that the spectral sequences are strongly convergent. So the maps $i$ and $j$ indeed lift to $BP\2$-module maps 
\[ V_{X} \underset{\widetilde{i}}{\rightarrow} BP\2 \smash X \underset{\widetilde{j}}{\rightarrow} V_{X} .\]
Note that $\widetilde{j} \circ \widetilde{i}$ induce an isomorphism on $H_{*}^{BP\2}$-homology. By \cref{whitehead}, this implies that $j \circ i$ is an equivalence up to $p$-completion. So indeed \[ BP\2 \smash X \simeq C_{X} \vee V_{X}\]
is a $BP\2$-module splitting.

\end{proof}

\subsection{Further analysis of the relative Adams spectral sequence for $BP\2_{*}X$}\label{sec:leftiso}
The goal of this section is to show that there are no hidden extensions in the Adams spectral sequences converging to $BP\2_{*}BP\2$ and $BP\2_{*}l_{k}$. 

Consider $BP\2 \smash X$, where $X$ is $BP\2$ or $l_{k}$. Recall from \cref{thm:dominic:splitting2} that $BP\2 \smash X$ splits as
\[ BP\2 \smash X \simeq C_{X} \vee V_{X},\]

where $V_{X}$ is a sum of suspensions mod-$p$ Eilenberg-Maclane spectra, and $\pi_{*}C_{X}$ is $v_{2}$-torsion free.
We will denote the summands $C_{BP\2}$ (resp. $V_{BP\2}$) by $C$(resp. $V$), and we will denote $C_{l_{k}}$ (resp. $V_{l_{k}}$) by $C_{k}$ (resp. $V_{k}$). 

Note that \[ BP\2_{*}X \simeq [BP\2, BP\2 \smash X]^{BP\2},\]
so there exists an Adams spectral sequence 
\[ Ext_{H_{*}^{BP\2}H}\big(H_{*}^{BP\2}BP\2, H_{*}^{BP\2}BP\2 \smash X \big) \Longrightarrow BP\2_{*}X.\]

Using \cref{lem:kunneth}, we can rewrite this as 
\[ Ext_{E(2)_{*}}\big(\F_{p}, H_{*}^{BP\2}X \big) \Longrightarrow BP\2_{*}X.\]

(Note that this is just the $E_{2}$-page of the ordinary Adams spectral sequence for $BP\2_{*}X$, after the classical change-of-rings). 

The splitting above implies that the Adams spectral sequence converging to $BP\2_{*}BP\2$ splits as the following two separate spectral sequences:

\begin{equation}\label{eqn:assC} E_{2}^{s,t} = Ext_{E(2)_{*}}^{s,t}\big(\F_{p}, H_{*}^{BP\2}C\big) \Longrightarrow \pi_{t-s}C  \end{equation}
and 
\begin{equation}\label{eqn:assV} E_{2}^{s,t} =  Ext_{E(2)_{*}}^{s,t}\big(\F_{p}, H_{*}^{BP\2}V \big) \Longrightarrow \pi_{t-s}V .\end{equation}

We already know that $\pi_{*}V \cong \oplus \Sigma^{?} H\F_{p}$, so no hidden extensions are possible in (\ref{eqn:assV}). However, showing that there are no hidden extensions in the spectral sequence converging to $\pi_{*}C$ is a lot more work. In \cite[Thm~2.1]{dominic_even} and \cite[Thm~3.19]{dominic_odd}, Culver showed that $Ext_{E(2)_{*}}(\F_{p}, H_{*}^{BP\2}C)$ is $v_{2}$-torsion free, that is, $v_{2}^{r}x \neq 0$ for all $x \neq 0 \in Ext_{E(2)_{*}}(\F_{p}, H_{*}^{BP\2}C)$ and $r \in \N$. While there is certainly torsion in the form of equations like $v_{0}x + v_{1}y + v_{2}z = 0$, we can adapt his proof to show that  $Ext_{E(2)_{*}}(\F_{p}, H_{*}^{BP\2}C)$ is $v_{0}$ and $v_{1}$-torsion free in the same manner (that is, $v_{0}^{r}x \neq 0$ and $v_{1}^{r}x \neq 0$ for all $x \neq 0 \in Ext_{E(2)_{*}}(\F_{p}, H_{*}^{BP\2}C)$ and $r \in \N$). This will imply that there is no room for hidden extensions in (\ref{eqn:assC}).

Recall that by the equivalence of categories of \cref{cor:1:equiv}, 
\[ Ext_{E(2)_{*}}(\F_{p}, H_{*}^{BP\2}C) \cong Ext_{E(2)}(\F_{p}, H_{*}^{BP\2}C).  \]

It will be easier to look at things from the $E(2)$-module perspective in this section, since we will use Margolis homology extensively later on. First we will discuss the $v_{i}$-Bockstein spectral sequences.
\begin{proposition}\label{prop:bss}
Let $M$ be an $E(2)$-module. Let $(i,j,h)$ be any permutation of $(0,1,2)$. There exists a $v_{i}$-Bockstein spectral sequence

\[ ^{v_{i}BSS}E_{1}^{*,*,*} = Ext_{E(Q_{j}, Q_{h})}\big(\F_{p}, M\big) [v_{i}] \Longrightarrow Ext_{E(2)}\big(\F_{p}, M\big)  \]

where \[^{v_{i}BSS}E_{1}^{s,t,r} = Ext_{E(Q_{j}, Q_{h})}^{s,t} \big( \F_{p},M \big)\{v_{i}^{r}\}\]

and the differential $d_{m}$ has the form 
\[d_{m}: E_{m}^{ s,t,r } \rightarrow E_{m}^{ s-k+1,t- |Q_{i}|k, r + k }.\] 

\end{proposition}

\begin{proof}
This can be proved in exactly the same way as Culver's \cite[Corollary~3.11]{dominic_odd}, substituting $v_{i}$ for $v_{2}$ and $E(Q_{j}, Q_{h})$ for $E(Q_{0}, Q_{1})$ in the appropriate places.
\end{proof}

We will use the $v_{i}$-Bockstein spectral sequence to show that there is no $v_{i}$-torsion in $Ext_{E(2)_{*}}^{s,t}(\F_{p}, H_{*}^{BP\2}C)$ if $Ext_{E(Q_{j}, Q_{h})}^{s,t}(\F_{p}, H_{*}^{BP\2}C)$ is concentrated in even $t-s$ degrees.

\begin{proposition}\label{bss:torsion}
Let $M$ be an $E(2)$-module. Suppose that $Ext_{E(Q_{j}, Q_{h})}^{s,t}(\F_{p}, M)$ is concentrated in even $t-s$ degrees. Then there is no $v_{i}$-torsion in $Ext_{E(2)}^{s,t}\big(\F_{p}, M\big)$.
\end{proposition}

\begin{proof}
Observe that the differential $d_{m}$ has $t-s$ degree $ (1-|Q_{i}|)k - 1$, which is always odd. So if $Ext_{E(Q_{j}, Q_{h})}^{s,t}\big(\F_{p}, M)\big)$ is concentrated in even $(t-s)$-degrees, then there are no nonzero differentials in the $v_{i}$-BSS, and so the spectral sequence collapses at the $E_{1}$-page. Thus $v_{i}^{n}x \neq 0$ for all $n \in \N$ and $x \neq 0 \in Ext_{E(2)}^{s,t}\big(\F_{p}, M \big)$.
\end{proof}

So if we can show that $Ext_{E(Q_{j}, Q_{h})}^{s,t}(\F_{p}, H_{*}^{BP\2}C)$ is concentrated in even $(t-s)$-degrees, then we will have proven that there is no $v_{i}$-torsion in $Ext_{E(2)}^{s,t}\big(\F_{p}, H_{*}^{BP\2}C\big)$. We will spend most of the remainder of this section proving that $Ext_{E(Q_{j}, Q_{h})}^{s,t}(\F_{p}, H_{*}^{BP\2}C)$ is indeed concentrated in even $(t-s)$-degrees.

In \cref{sec:3.5.1}, we will produce a splitting 
\[ H_{*}^{BP\2}C \cong_{E(Q_{j}, Q_{h})} S_{i} \oplus R_{i} \]
such that $S_{i}$ is a free $E(Q_{j}, Q_{h})$-module and and $R_{i}$ has no free summands, for each permutation $(i,j,h)$ of $\{0,1,2\}$ (\cref{split:Ri}). It follows quickly that $Ext_{E(Q_{j}, Q_{h})}(\F_{p}, S_{i})$ is concentrated in even $t-s$ degree. This reduces the problem to showing that $Ext_{E(Q_{j}, Q_{h})}(\F_{p}, R_{i})$ is also concentrated in even degrees. In \cref{sec:classification}, we will recall Adams-Priddy's usage of Margolis homology to classify invertible modules over an exterior algebra on two generators. In \cref{sec:ext:inv}, we will recall Adams-Priddy's computation of the Ext groups of these invertible modules. Next we will analyze the Margolis homology of particular submodules of $H_{*}BP\2$: in \cref{sec:marg:BP2:odd}, we will handle the odd-primary case, and in \cref{sec:marg:BP2:even}, we will present the analogous results at the prime $2$. These computations will allow us to decompose $H_{*}BP\2$ in terms of invertible modules. In \cref{sec:ext:analysis}, we will combine the results of \cref{sec:classification} and \cref{sec:ext:inv} with our computations from \cref{sec:marg:BP2:odd} and \cref{sec:marg:BP2:even} to show that $Ext_{E(Q_{j}, Q_{h})}(\F_{p}, R_{i})$ is indeed concentrated in even $t-s$ degrees. In \cref{sec:nhe}, we will combine these computations with \cref{bss:torsion} to conclude that there are no hidden extensions in the Adams spectral sequence converging to $BP\2_{*}BP\2$, as well as the one converging to $BP\2_{*}l_{k}$ for all $k \in \N$.

\subsubsection{Decomposing $H_{*}^{BP\2}C$}\label{sec:3.5.1}
We start by recalling the notion of length.

\begin{definition}\cite[Definition~3.1]{dominic_odd}
Let $x$ be a monomial in $A_{*}$.
If $p$ is odd, then we can write $x$ in the form 
\[ x = \bar{\xi}_{1}^{i_{1}}\bar{\xi}_{2}^{i_{2}}\cdots \bar{\xi}_{r}^{i_{r}}\bar{\tau}_{0}^{\epsilon_{0}}\bar{\tau}_{1}^{\epsilon_{1}}\cdots \bar{\tau}_{s}^{\epsilon_{s}} \in A_{*}.\]
Then the \textit{length} of the monomial, denoted $\ell(x)$, is defined to be \[\ell(x) = \epsilon_{0} + \epsilon_{1}  + \cdots + \epsilon_{s}.\]

If $p=2$, then we can write $x$ in the form 
\[ x = \bar{\xi}_{1}^{2j_{1} + \epsilon_{1}}\bar{\xi}_{2}^{j_{2}+ \epsilon_{2}}\cdots \bar{\xi}_{r}^{j_{r} + \epsilon_{r}},\]
where $0 \le \epsilon_{i} \le 1$.
Then the \textit{length} of the monomial, denoted $\ell(x)$, is defined to be \[\ell(x) = \epsilon_{1} + \epsilon_{2}  + \cdots + \epsilon_{r}.\]

\end{definition}

\begin{remark}\label{remark:length}
Consider a monomial $x$. Note that the length of every nonzero summand of $Q_{j}x$ is exactly $l(x)-1$.  
\end{remark}

\begin{proposition}\cite[Proposition~3.7,~Corollary~3.22]{dominic_odd}\cite[Proposition~2.22, Section~2.6]{dominic_even}

Let $S$ be the $E(2)$-submodule of $H_{*}BP\2$ generated by monomials of length at least $3$, ie

\[ S:= E(2)\{ x \in H_{*}BP\2| \ \ell(x) \ge 3  \}. \]
Then
\[H_{*}^{BP\2}V \cong S .\]
\end{proposition}

So if $x$ is a monomial in $H_{*}^{BP\2}C$, then $l(x) \le 2$. 

\begin{proposition}
Let $(i,j,h)$ be a permutation of $\{0,1,2\}$. Let $S_{i}$ be the $E(Q_{j}, Q_{h})$-submodule of $H_{*}^{BP\2}C$ generated by monomials of length exactly $2$, that is, 
\[ S_{i}:= E(Q_{j}, Q_{h})\{ x \in H_{*}BP\2| \ \ell(x) = 2  \}. \]

Then $S_{i}$ is a free $E(Q_{j}, Q_{h})$ module.
\end{proposition}

\begin{proof}
This is completely analogous to the proof of Proposition 3.12 of \cite{dominic_odd} (for odd primes) and Corollary 2.24 of \cite{dominic_even} (for $p=2$). To prove the $S_{i}$ case, substitute $E(Q_{j}, Q_{h})$ for $E(Q_{0}, Q_{1})$.
\end{proof}

\begin{corollary}\label{split:Ri}
Let $R_{i}$ be the quotient of $H_{*}C$ by $S_{i}$, that is $R_{i}$ is an $E(Q_{j}, Q_{h})$ module such that the following sequence is exact:

\[ 0 \rightarrow S_{i} \rightarrow H_{*}C \rightarrow R_{i} \rightarrow 0 .\]

Then there exists a splitting of $E(Q_{j}, Q_{h})$-modules

\[ H_{*}C \cong_{E(Q_{j}, Q_{h})} S_{i} \oplus R_{i}.\]

Furthermore, the summand $R_{i}$ contains no free summands.
\end{corollary}

\begin{proof}
The module $S_{i}$ is free, so \cref{free:proj:inj} tells us that $S_{i}$ is also injective. So the exact sequence splits. Consider a monomial $x \in R_{i}$. The length of $x$ is at most one, so \cref{remark:length} tells us that $Q_{j}Q_{h}x = 0$. So there are no monomials that could generate a free summand in $R_{i}$.
\end{proof}

Now we will analyze $Ext_{E(Q_{j}, Q_{h})}^{s,t}(\F_{p}, S_{i})$.

\begin{proposition}\label{ext:Si}
The term $Ext_{E(Q_{j}, Q_{h})}^{s,t}\big( \F_{p}, S_{i} \big)$ is concentrated in degrees $(s,t)$ such that $s=0$ and $t$ is even.
\end{proposition}

\begin{proof}
Recall that ${S_{i}} \cong \bigoplus\limits_{x \in J} E(Q_{j}, Q_{h})\{x\}$, where $J$ is the set of monomials in $H_{*}^{BP\2}C$ having length exactly $2$. It follows that \[Ext_{E(Q_{j}, Q_{h})}^{s,t}(\F_{p}, S_{i}) \cong \bigoplus\limits_{x \in J} \Sigma^{|Q_{0}Q_{1}x|}\F_{p}.\] Monomials of even length have even dimension, so $|Q_{0}Q_{1}x|$ is also even. So indeed ${Ext_{E(Q_{j}, Q_{h})}^{s,t}\big( \F_{p}, S_{i} \big)}$ is concentrated in degrees $(s,t)$ such that $s= 0$ and $t$ is even.
\end{proof}

In order to show that $Ext^{s,t}_{E(Q_{j}, Q_{h}})\big( \F_{p}, R_{i} \big)$ is also concentrated in even degrees, we will first need to recall some results of Adams and Priddy on the classification of invertible modules over an exterior algebra on two generators. 

\subsubsection{The classification of invertible modules}\label{sec:classification}

We will start by recalling Margolis's stable category of modules over an exterior algebra \cite[p.205]{margolis_book}. 
Let $M,N$ be graded modules over an exterior algebra $E$. A homomorphism $f:M \rightarrow N$ is said to be \textit{stably trivial} if there exists a projective $E$-module $P$ such that $f$ factors through $P$.
A pair of homomorphisms $g,h: M \rightarrow N$ is said to be \textit{stably equivalent} (denoted $g\sim h$) if $g-h$ is stably trivial.

\begin{definition}\cite[p.205]{margolis_book}
Let $Stab(E)$ denote the stable category of $E$-modules. That is, $Stab(E)$ is the category whose objects are the $E$-modules, and whose morphisms are $Hom_{E}(M,N)/\sim$.
\end{definition}

An equivalence class $M$ in $Stable(E)$ is said to be \textit{invertible} if there exists an equivalence class $N$ such that $M \otimes N$ is stably equivalent to $\F_{p}$. The invertible stable equivalence classes and their tensor product form the Picard group of $Stable(E)$, denoted $Pic \big( Stable(E) \big)$.

\begin{theorem}\cite[Theorem~3.6]{adams_priddy}\label{pic}
Let $E(\alpha, \beta)$ be a graded exterior algebra on two generators over $\F_{p}$ such that $|\alpha| < |\beta|$. Let $\Sigma$ denote the graded $E(\alpha, \beta)$-module such that \[
\Sigma_{t} = \begin{cases}
\F_{p} & \text{ if $t=1$}\\
0 & \text{ if $t \neq 1$}\\
\end{cases}
.\]
Let $I$ denote the augmentation ideal of $E(\alpha, \beta)$.
Then the Picard group $Pic\Big(Stable\big(E(\alpha, \beta \big)\Big)$ is 
\[Pic\Big(Stable\big(E(\alpha, \beta \big)\Big) \cong \Z\{\Sigma\} \oplus \Z\{I\}.\]
\end{theorem}

We can use Margolis homology to identify when a module is invertible.

\begin{definition}
Let $M$ be a module over $E(\alpha, \beta)$. The Margolis homology of $M$ with respect to $\alpha$ is denoted by $\M_{*}(M, \alpha)$, and is defined to be \[\M_{*}(M, \alpha) = \frac{ker\ \alpha:M \rightarrow M}{im\ \alpha:M \rightarrow M}.\]
\end{definition}

Margolis homology has a Whitehead Theorem and a K\"unneth Theorem, both of which we will use in this section.

\begin{proposition}(K\"unneth Theorem)\cite[p.315]{margolis_book}\label{marg:kunneth}
Let $M$ and $N$ be $E(\alpha, \beta)$-modules. Then
\[ \M_{*}(M \otimes N, \alpha) \cong \M_{*}(M, \alpha) \otimes \M_{*}(N, \alpha) .\]
\end{proposition}

\begin{proposition}(Whitehead Theorem)\cite[Thm~18.3]{margolis_book}\label{marg:whitehead}
Let $M$ and $N$ be connective $E(\alpha, \beta)$-modules. Consider an $E(\alpha, \beta)$-module homomorphism $f: M \rightarrow N$. The map $f$ is a stable equivalence if and only if $f$ induces an isomorphism on both $\alpha$ and $\beta$-Margolis homology.
\end{proposition}

\begin{lemma}\label{inv:marg}\cite[Proposition~3.5]{adams_priddy} 
Consider a graded exterior algebra on two generators $E(\alpha, \beta)$ over a field $\F_{p}$. Let $M$ be a finitely generated graded module over $E(\alpha, \beta)$. Then $M$ is stably equivalent to an invertible module if and only if both $\M_{*}(M, \alpha)$ and $\M_{*}(M, \beta)$ are one-dimensional vector spaces over $\F_{p}$.
\end{lemma}

So if a finitely generated graded module $M$ has one-dimensional Margolis homology, then $M \cong_{E(\alpha, \beta)}\Sigma^{a} \otimes I^{\otimes b} \oplus F $, where $a,b \in \Z$ and $F$ is a free module. Furthermore, we can use the Margolis homology of $M$ to determine $a$ and $b$. 

Note that the tensor product $\Sigma \otimes M$ produces an $E(\alpha, \beta)$ module which is identical to $M$, except that each element of $\Sigma \otimes M$ is one degree higher than the corresponding element of $M$. So we will denote $\Sigma^{a} \otimes M$ by $\Sigma^{a}M$.

\begin{proposition}\label{marg:I}\cite[Pf.~of~Theorem~3.6]{adams_priddy}
The Margolis homology of $\Sigma^{a}I^{\otimes b}$ is 
\[ \M_{*}(\Sigma^{a}I^{\otimes b}, \alpha) \cong \Sigma^{a + |\alpha| b}\F_{p} \]
\[ \M_{*}(\Sigma^{a}I^{\otimes b}, \beta) \cong \Sigma^{a + |\beta|b}\F_{p} .\]
\end{proposition}

\begin{proof}
The algebra $E(\alpha, \beta)$ acts trivially on the module $\Sigma$, so the Margolis homology of $\Sigma$ is 
\[ \M_{*}(\Sigma, \alpha) \cong \Sigma \F_{p} \]
\[ \M_{*}(\Sigma, \beta) \cong \Sigma \F_{p} .\]

Consider the long exact sequence
\[ \cdots \rightarrow M_{*}\big(E(\alpha, \beta), \alpha\big) \rightarrow \M_{*}\big(\F_{p}, \alpha\big) \rightarrow \Sigma^{-|\alpha|}\M_{*}\big(I, \alpha\big) \rightarrow \Sigma^{-|\alpha|}M_{*}\big(E(\alpha, \beta), \alpha\big) \rightarrow  \cdots .\]

Since the Margolis homology of $E(\alpha, \beta)$ is trivial, it follows that the Margolis homology of the augmentation ideal $I$ is
\[ \M_{*}(I, \alpha) \cong \Sigma^{|\alpha|} \F_{p} \]
\[ \M_{*}(I, \beta) \cong \Sigma^{|\beta|} \F_{p} .\]
By the K\"unneth formula (\cref{marg:kunneth}),
\[ \M_{*}(\Sigma^{a}I^{\otimes b}, \alpha) \cong \Sigma^{a + |\alpha|b}\F_{p} \]
\[ \M_{*}(\Sigma^{a}I^{\otimes b}, \beta) \cong \Sigma^{a + |\beta|b}\F_{p} .\]
\end{proof} 

It will be useful to have an explicit description of $I^{-1}$.

\begin{lemma}\label{def:J} \cite[Proof~of~Lemma~3.5]{adams_priddy}
Let $J$ be the module defined by the short exact sequence 
\begin{equation}\label{ses:J} 0 \rightarrow \F_{p} \rightarrow E(\alpha, \beta) \rightarrow J \rightarrow 0 .
\end{equation}
Then $J = I^{-1}$. 
\end{lemma}

\begin{proof}
Applying Margolis homology to the short exact sequence (\ref{ses:J}) induces a long exact sequence 
\[ \cdots \rightarrow \Sigma^{|\alpha|}\M_{*}\big(J, \alpha\big) \rightarrow \M_{*}\big(\F_{p}, \alpha\big) \rightarrow \M_{*}\big(E(\alpha, \beta), \alpha\big) \rightarrow \cdots .\]
Recall that $\M_{*}\big(E(\alpha, \beta), \alpha\big) = 0$. So \[\Sigma^{|\alpha|}\M_{*}\big(J, \alpha\big) \cong \M_{*}\big(\F_{p}, \alpha\big).\]
Thus $\M_{*}\big(J, \alpha\big) \cong \Sigma^{-|\alpha|} \F_{p}$. By the same argument, $\M_{*}\big(J, \beta \big) \cong \Sigma^{-|\beta|} \F_{p}$. It follows from \cref{inv:marg} that $J$ is an invertible module (up to stable equivalence). Observe that $J$ is a quotient of $E(\alpha, \beta)$, so it has no free summands. By \cref{pic}, $J$ is isomorphic to $\Sigma^{a}I^{b}$ for some integers $a$ and $b$. It follows from \cref{marg:I} that $a = 0$ and $b=-1$. So indeed $J$ is $I^{-1}$.
\end{proof} 

\subsubsection{The $Ext$ groups of invertible modules}\label{sec:ext:inv}
Now we will recall some results on the $Ext$ groups of invertible modules.
\begin{corollary}\label{ext:J}
If $b \le 0$, then \[ Ext_{E(\alpha, \beta)}^{s,t}\big( \F_{p}, I^{\otimes b} \big) \cong Ext_{E(\alpha, \beta)}^{s - b,t}\big( \F_{p}, \F_{p} \big) \text{ for $s > 0$}  .\]
\end{corollary}

\begin{proof}
The short exact sequence (\ref{ses:J}) induces a long exact sequence

\[ \cdots \rightarrow Ext_{E(\alpha, \beta)}^{s}\big( \F_{p}, E(\alpha, \beta)\big) \rightarrow Ext_{E(\alpha, \beta)}^{s}\big( \F_{p}, J \big) \rightarrow Ext_{E(\alpha, \beta)}^{s+1}\big( \F_{p}, \F_{p} \big) \rightarrow \cdots .\]

By \cref{free:proj:inj}, the module $E(\alpha, \beta)$ is injective, so $Ext_{E(\alpha, \beta)}^{s}\big( \F_{p}, E(\alpha, \beta)\big) = 0$ for all $s > 0$. So 

\[Ext_{E(\alpha, \beta)}^{s}\big( \F_{p}, J \big) \cong Ext_{E(\alpha, \beta)}^{s+1}\big( \F_{p}, \F_{p} \big)\]

for all $s > 0$. Tensoring the short exact sequence (\ref{ses:J}) with $J^{\otimes k}$ yields a short exact sequence 
\[ 0 \rightarrow J^{\otimes k} \rightarrow F \rightarrow J^{\otimes k+1} \rightarrow 0 ,\]
where $F := E(\alpha, \beta) \otimes J^{\otimes k}$ is a free module. So by induction,

\[Ext_{E(\alpha, \beta)}^{s}\big( \F_{p}, J^{\otimes k+1} \big) \cong Ext_{E(\alpha, \beta)}^{s+1}\big( \F_{p}, J^{\otimes k} \big)\]

for $s > 0$, for all $k \in \N$. That is, for all $k \ge 0$
\[Ext_{E(\alpha, \beta)}^{s}\big( \F_{p}, J^{\otimes k} \big) \cong Ext_{E(\alpha, \beta)}^{s+b}\big( \F_{p}, \F_{p} \big) \text { for $s>0$}.\]

By \cref{def:J}, it follows that for all $b \le 0$,
\[Ext_{E(\alpha, \beta)}^{s}\big( \F_{p}, I^{\otimes b} \big) \cong Ext_{E(\alpha, \beta)}^{s-b}\big( \F_{p}, \F_{p} \big) \text { for $s>0$}.\]
\end{proof}

We will use the following criteria for showing that the $Ext^{s,t}$ groups of a module over $E(Q_{j}, Q_{h})$ are zero when $s > 0$ and $(t-s)$ is odd. 

\begin{proposition}\label{marg:aux}
Consider a graded $E(Q_{j}, Q_{h})$-module $M$. Suppose without loss of generality that $j < h$. If the Margolis homology of $M$ is
\[ \M_{*}\big( M, Q_{j} \big) \cong \F_{p}\{ x \} \]
\[ \M_{*}\big( M, Q_{h} \big) \cong \F_{p}\{ y \} \]
where $|x|$ is even and $|y| > |x| > 0$, 
then $Ext_{E(Q_{j}, Q_{h})}^{s,t}\big(\F_{p}, M\big)$ is concentrated in even $(t-s)$-degrees for all $s > 0$. 
\end{proposition}

\begin{proof}
By \cref{inv:marg} and \cref{pic}, the module $M$ is stably equivalent to $\Sigma^{a}I^{b}$ for some $a,b \in \N$. By \cref{marg:I}, 

\[ a + b|Q_{j}| = |x| \]
\[ a + b|Q_{h}| = |y| .\]

It follows that 
\begin{equation}\label{eqn:b} b(|Q_{h}| - |Q_{j}|) = |y|-|x| \end{equation}
\begin{equation}\label{eqn:a} a = |x| - b|Q_{j}|. \end{equation}

Note that $|Q_{h}| < |Q_{j}| < 0$, so $|Q_{h}| - |Q_{j}|$ is negative. Since $|y|-|x|$ is positive, \cref{eqn:b} implies that $b$ must also be negative. So we can apply \cref{ext:J} to get 
\[Ext_{E(Q_{j}, Q_{h})}^{s,t}\big( \F_{p}, \Sigma^{a}I^{b} \big) \cong Ext_{E(Q_{j}, Q_{h})}^{s-b,t-a}\big( \F_{p}, \F_{p} \big) \text{ for all $s>0$}.\] 

Recall that \[Ext_{E(Q_{j}, Q_{h})}^{s,t}\big( \F_{p}, \F_{p} \big) \cong \F_{p}[v_{j}, v_{h}].\] 

Note that $|v_{j}|_{(s,t)} = (1, 2p^{j}-1)$, so $Ext_{E(Q_{j}, Q_{h})}^{s,t}\big( \F_{p}, \F_{p} \big)$ is concentrated in even $(t-s)$-degrees.

Furthermore, since $|x|$ is even and $|Q_{j}|$ is odd, \cref{eqn:a} tells us that $b$ has the same parity as $a$.

Thus $Ext_{E(Q_{j}, Q_{h})}^{s,t}\big( \F_{p}, \Sigma^{a}I^{b} \big)$ must also be concentrated in even $t-s$ degrees. Recall that \[Ext_{E(Q_{j}, Q_{h})}^{s,t}\big( \F_{p}, M \big) \cong Ext_{E(Q_{j}, Q_{h})}^{s,t}\big( \F_{p}, \Sigma^{a}I^{b} \big) \]
for all $s > 0$, so indeed $Ext_{E(Q_{j}, Q_{h})}^{s,t}\big( \F_{p}, M \big) $ is concentrated in even $(t-s)$-degrees for all $s>0$.
\end{proof}

Next we will analyze the Margolis homology of $H_{*}BP\2$ and use it to decompose $H_{*}BP\2$ in terms of invertible $E(Q_{j}, Q_{h})$
-modules. The description differs slightly at $p=2$ from other primes, so we will first describe the odd-primary situation and then do $p=2$.
\subsubsection{Analyzing the Margolis homology of $H_{*}BP\2$: the odd-primary case}\label{sec:marg:BP2:odd}

Let $W_{1}$ denote the $E(Q_{1}, Q_{2})$-subalgebra of $H_{*}BP\2$ generated by the exterior elements, that is,

\[ W_{1} = E(\bar{\tau}_{3}, \bar{\tau}_{4},\bar{\tau}_{5},\ldots ) \otimes \F_{p}[ \bar{\xi}_{1}^{p^{2}}] \otimes \F_{p}[ \bar{\xi}_{2}^{p}, \bar{\xi}_{3}^{p}, \ldots ] .\]

Let $W_{o}$ be the $E(Q_{0}, Q_{2})$-subalgebra of $H_{*}BP\2$ generated by exterior elements $\bar{\tau}_{k}$ such that $k$ is odd, that is,
\[ W_{o} = E(\bar{\tau}_{3}, \bar{\tau}_{5},\ldots ) \otimes \F_{p}[\bar{\xi}_{1}^{p^{2}}] \otimes \F_{p}[ \bar{\xi}_{3}, \bar{\xi}_{5}, \ldots ] .\]

Likewise, let $W_{e}$ be the $E(Q_{0}, Q_{2})$-subalgebra of $H_{*}BP\2$ generated by exterior elements $\bar{\tau}_{k}$ such that $k$ is even:
\[ W_{e} = E(\bar{\tau}_{4}, \bar{\tau}_{6},\cdots ) \otimes \F_{p}[\bar{\xi}_{2}^{9}] \otimes \F_{p}[ \bar{\xi}_{4}, \bar{\xi}_{6}, \ldots ] .\]

Let $T_{k}(x_{1}, x_{2},\ldots ) = \F_{p}[x_{1}, x_{2}, \ldots]\Big{/}(x_{1}^{p^{k}}, x_{2}^{p^{k}},\ldots)$. 

The following lemma is an adaptation of Lemma 3.15 of \cite{dominic_odd}. 

\begin{lemma}\label{lem:2.30}
As an $E(Q_{1}, Q_{2})$-module,
\[ H_{*}BP\2 \cong W_{1} \otimes T_{2}(\bar{\xi}_{1}) \otimes T_{1}(\bar{\xi}_{2}, \bar{\xi}_{3}, \ldots) .\] 

As an $E(Q_{0}, Q_{2})$-module, 
\[ H_{*}BP\2 \cong W_{e} \otimes W_{o} \otimes T_{2}(\bar{\xi}_{1},\bar{\xi}_{2}) .\] 

\end{lemma}

We will need to use the Margolis homology of $H_{*}BP\2$.

\begin{theorem}\cite[Theorem~2.18]{dominic_odd}\label{marg:BP2}
The Margolis homology of $H_{*}BP\2$ is
\[ \M_{*}(H_{*}BP\2, Q_{0}) \cong \F_{p}[\bar{\xi_{1}}, \bar{\xi_{2}}] \]
\[ \M_{*}(H_{*}BP\2, Q_{1}) \cong \F_{p}[\bar{\xi_{1}}] \otimes T_{1}(\bar{\xi_{2}}, \bar{\xi_{3}}, \ldots) \]
\[ \M_{*}(H_{*}BP\2, Q_{2}) \cong T_{2}(\bar{\xi_{1}},\bar{\xi_{2}}, \bar{\xi_{3}}, \ldots) .\]
\end{theorem}

Next we will determine the Margolis homology of $W_{1}$.

\begin{proposition}\label{marg:R1}
The Marolis homology of $W_{1}$ is

\[ \M_{*}\big( W_{1}, Q_{1} \big) \cong \F_{p}[\bar{\xi}_{1}^{p^{2}}] \]
\[ \M_{*}\big( W_{1}, Q_{2} \big) \cong T_{1}( \bar{\xi}_{2}^{p}, \bar{\xi}_{3}^{p}, \ldots  ).\]
\end{proposition}

\begin{proof}
By the K\"unneth formula for Margolis homology (\cref{marg:kunneth}),  
\[\M_{*}\big( H_{*}BP\2, Q_{1} \big) \cong \M_{*}\big( W_{1}, Q_{1} \big) \otimes \M_{*}\big( T_{2}(\bar{\xi}_{1}), Q_{1} \big) \otimes \M_{*}\big( T_{1}(\bar{\xi}_{2}, \bar{\xi}_{3,}, \ldots), Q_{1} \big).\]
Since the $Q_{1}$-action on $T_{2}(\bar{\xi}_{1})$ and $T_{1}(\bar{\xi}_{2}, \bar{\xi}_{3,}, \ldots)$ is trivial, it follows that \[\M_{*}\big( T_{2}(\bar{\xi}_{1}), Q_{1} \big) \cong T_{2}(\bar{\xi}_{1})\] \[ \M_{*}\big( T_{1}(\bar{\xi}_{2}, \bar{\xi}_{3}, \ldots), Q_{1} \big)   \cong T_{1}(\bar{\xi}_{2}, \bar{\xi}_{3}, \ldots).\] 

Combining this with \cref{marg:BP2}, we see that \[ \M_{*}\big( H_{*}BP\2, Q_{1} \big) \cong \M_{*}\big( W_{1}, Q_{1} \big) \otimes  T_{2}(\bar{\xi}_{1}) \otimes  T_{1}(\bar{\xi}_{2}, \bar{\xi}_{3}, \ldots)  \cong \F_{p}[\bar{\xi}_{1}] \otimes T_{1}(\bar{\xi}_{2}, \bar{\xi}_{3}, \ldots) . \]

It follows that
\[ \M_{*}\big( W_{1}, Q_{1} \big) \cong \F_{p}[\bar{\xi}_{1}^{p^{2}}] .\]

Likewise, note that the $Q_{2}$-action on $T_{2}(\bar{\xi}_{1})$ and $T_{1}(\bar{\xi}_{2}, \bar{\xi}_{3,}, \ldots)$ is trivial. So by the same argument, 
\[ \M_{*}\big( W_{1}, Q_{2} \big) \cong T_{1}(\bar{\xi}_{2}^{p}, \bar{\xi}_{3}^{p}, \ldots ) .\]

\end{proof}

Note that both the $Q_{0}$ and $Q_{2}$-actions on $T_{2}(\bar{\xi}_{1}) \otimes T_{1}(\bar{\xi}_{3}, \bar{\xi}_{5}, \ldots)$ and on $T_{2}(\bar{\xi}_{2}) \otimes T_{1}(\bar{\xi}_{4}, \bar{\xi}_{6}, \ldots)$ are trivial. So we can use the same argument as in the proof of \cref{marg:R1} to compute the Margolis homology of $W_{o}$ and $W_{e}$.

\begin{proposition}\label{marg:ReRo}
The Margolis homology of $W_{e}$ and $W_{o}$ is 
\[\M_{*}\big(W_{e},Q_{0} \big) \cong  \F_{p}[\bar{\xi}_{2}^{p^{2}}]\]
\[\M_{*}\big(W_{e},Q_{2} \big) \cong T_{2}( \bar{\xi}_{4}, \bar{\xi}_{6}, \ldots )  \]
\[\M_{*}\big(W_{o},Q_{0} \big) \cong  \F_{p}[\bar{\xi}_{1}^{p^{2}}]\]
\[\M_{*}\big(W_{o},Q_{2} \big) \cong T_{2}( \bar{\xi}_{3}, \bar{\xi}_{5}, \ldots ) . \]
\end{proposition}

Note that all monomials in $W_{1}$ have weight divisible by $p^{3}$. Let $W_{1}(k)$ denote the weight $p^{3}k$ component of $W_{1}$. The actions of $Q_{1}$ and $Q_{2}$ on $H_{*}BP\2$ are weight-preserving, so \begin{equation}\label{marg:decomp:R1}
\M_{*}(W_{1}, Q_{j}) \cong \bigoplus\limits_{k=0}^{\infty} \M_{*}( W_{1}(k), Q_{j} )
\end{equation}

for $j=1,2$.

Likewise, the actions of $Q_{0}$ and $Q_{2}$ on $H_{*}BP\2$ are weight-preserving, and every monomial in $W_{o}$ is divisible by $p^{3}$, so we will let $W_{o}(n)$ denote the weight $p^{3}n$ component of $W_{o}$. Note that every monomial in $W_{e}(n)$ is divisible by $p^{4}$, so it will be more convenient to let $ W_{e}(n)$ denote the $p^{4}n$ component of $W_{e}(n)$. Then we can use the decompositions 
\begin{equation}\label{marg:decomp:Re}\M_{*}(W_{e}, Q_{j}) \cong \bigoplus\limits_{n=0}^{\infty} \M_{*}( W_{e}(n), Q_{j} )\end{equation}

\begin{equation}\label{marg:decomp:Ro} \M_{*}(W_{o}, Q_{j}) \cong \bigoplus\limits_{n=0}^{\infty} \M_{*}( W_{o}(n), Q_{j} )\end{equation}
for $i=0,2$.\\

Next we will compute the Margolis homology of these submodules.

\begin{proposition}\label{marg:w1:n}
The Margolis homology of $W_{1}(n)$ is

\[ \M_{*}(W_{1}(n), Q_{1}) \cong
\F_{p}\{ \bar{\xi}_{1}^{p^{2}n} \} 
.\]

\[ \M_{*}(W_{1}(n), Q_{2}) \cong
\F_{p}\{ (\bar{\xi}_{2}^{n_{0}}\bar{\xi}_{3}^{n_{1}}\cdots \bar{\xi}_{r+2}^{n_{r}})^{p} \} \]

where  $n_{0} + n_{1}p + \cdots + n_{r}p^{r}$ is the $p$-adic expansion of $n$. 
\end{proposition}

\begin{proof}
The decomposition (\ref{marg:decomp:R1}) tells us that the Margolis homology of $W_{1}(n)$ should consist of the weight $p^{3}n$ component of the Margolis homology of $W_{1}$. 
Note that $wt(\bar{\xi}_{1}^{p^{2}m}) = p^{3}m$. Combining this with \cref{marg:R1}, we see that 
\[ \M_{*}(W_{1}(n), Q_{1}) \cong
\F_{p}\{ \bar{\xi}_{1}^{p^{2}n}\}. \]

Suppose that $x$ is a nonzero monomial in $\M_{*}(W_{1}, Q_{2})$. Then $x$ has the form \[x = \bar{\xi}_{2}^{m_{0}p}
\bar{\xi}_{3}^{m_{1}p}\cdots \bar{\xi}_{r+2}^{m_{r}p}\] where $0 \le m_{i} < p$ for all $0 \le i \le r$. 

Note that \[wt\big( \bar{\xi}_{2}^{m_{0}p}
\bar{\xi}_{3}^{m_{1}p}\cdots \bar{\xi}_{r+2}^{m_{r}p} \big) = p^{3}( m_{0} + m_{1}p + \ldots m_{r}p^{r} ),\]
so the weight of $x$ is exactly $p^{3}m$ where $m_{0} + m_{1}p + \cdots + m_{r}p^{r}$ is the $p$-adic expansion of $m$. 

So by \cref{marg:decomp:R1} and \cref{marg:R1}, 
\[ \M_{*}(W_{1}(n), Q_{2}) \cong
\F_{p}\{ (\bar{\xi}_{2}^{n_{0}}\bar{\xi}_{3}^{n_{1}}\cdots \bar{\xi}_{r+2}^{n_{r}})^{p} \} \]

where  $n_{0} + n_{1}p + \cdots + n_{r}p^{r}$ is the $p$-adic expansion of $n$. 
\end{proof}

\begin{proposition}\label{marg:W:eo:n}
The $Q_{0}$-Margolis homologies of $W_{o}(n)$ and $W_{e}(n)$ are
\[ \M_{*}\big(W_{o}(n), Q_{0} \big) \cong
 \F_{p}\{ \bar{\xi}_{1}^{p^{2}n} \}  
\]

\[ \M_{*}\big(W_{e}(n), Q_{0} \big) \cong
\F_{p}\{ \bar{\xi}_{2}^{p^{2}n} \}
.\]

Suppose that $n \in \N$, with $p$-adic expansion
$n = n_{0} + n_{1}p + \cdots n_{r}p^{r}$. Let  $m_{i} = n_{2i} + pn_{2i+1}$, and $s = \lfloor r/2 \rfloor$. Then
\[ \M_{*}\big(W_{e}(n), Q_{2} \big) \cong
 \F_{p}\{ \bar{\xi}_{4}^{m_{0}}\bar{\xi}_{6}^{m_{1}} \cdots \bar{\xi}_{2s+4}^{m_{s}} \}.
\]

Likewise,
\[ \M_{*}\big(W_{o}(n), Q_{2} \big) \cong
 \F_{p}\{ \bar{\xi}_{3}^{m_{0}}\bar{\xi}_{5}^{m_{1}} \cdots \bar{\xi}_{2s+3}^{m_{s}} \}
.\] 

\end{proposition}

\begin{proof}
First note that $wt( \bar{\xi}_{1}^{p^{2}n}) = p^{3}n$ and $wt( \bar{\xi_{2}}^{p^{2}n}) = p^{4}n$, so it follows immediately from combining \cref{marg:ReRo} with (\ref{marg:decomp:Re}) and (\ref{marg:decomp:Ro}) that indeed \[ \M_{*}\big(W_{o}(n), Q_{0} \big) \cong
 \F_{p}\{ \bar{\xi}_{1}^{p^{2}n} \}  
\]

\[ \M_{*}\big(W_{e}(n), Q_{0} \big) \cong
\F_{p}\{ \bar{\xi}_{2}^{p^{2}n} \}
.\]

Consider a monomial $x \in \M_{*}(W_{e}, Q_{2})$. It follows from \cref{marg:ReRo} that $x$ is of the form 
\[x = \bar{\xi}_{4}^{m_{0}}\bar{\xi}_{6}^{m_{1}} \cdots \bar{\xi}_{2r+4}^{m_{s}}, \]

where $ 0 \le m_{i} < p^{2}$ for all $0 \le i \le s$. Note that $wt(x)$ must be divisible by $p^{4}$, so $wt(x) = p^{4}n$ for some $n \in \N$. So 
\[ m_{0} + m_{1}p^{2} + \cdots + m_{s}p^{2s} = n .\]

Suppose that $n$ has $p$-adic expansion $n = n_{0} + n_{1}p + \cdots n_{r}p^{r}$. Then $m_{i} = n_{2i} + pn_{2i+1}$. So there is exactly one monomial of weight $p^{4}n$ for each $n$, and

\[ \M_{*}\big(W_{e}(n), Q_{2} \big) \cong
 \F_{p}\{ \bar{\xi}_{4}^{m_{0}}\bar{\xi}_{6}^{m_{1}} \cdots \bar{\xi}_{2s+4}^{m_{s}} \}.
\]

Likewise, 
\[ wt(\bar{\xi}_{3}^{m_{0}}\bar{\xi}_{5}^{m_{1}} \cdots \bar{\xi}_{2s+3}^{m_{s}}) = p^{3}n, \]

and so
\[ \M_{*}\big(W_{o}(k), Q_{2} \big) \cong
 \F_{p}\{ \bar{\xi}_{3}^{m_{0}}\bar{\xi}_{5}^{m_{1}} \cdots \bar{\xi}_{2s+3}^{m_{s}} \}
.\]

\end{proof}

\subsubsection{Analyzing the Margolis homology of $H_{*}BP\2$: the $2$-primary case}\label{sec:marg:BP2:even}

Here, we record the $2$-primary analogues of the results in \cref{sec:marg:BP2:odd}. Each computation and proof is exactly analogous to its odd-primary version, so our exposition is significantly less detailed than in the previous subsection. 

Recall from \cref{action:even} that at the prime $2$,

\[ H_{*}BP\2 \cong \F_{p}[\bar{\xi}_{1}^{2}, \bar{\xi}_{2}^{2}, \bar{\xi}_{3}^{2}, \bar{\xi}_{4}, \ldots] .\]

At the prime $2$, the action of $E(2)$ on $H_{*}BP\2$ is given by
\[Q_{j}\bar{\xi}_{k} =  \begin{cases}  \bar{\xi}_{k-j-1}^{2^{j+1}} & \text{ if } j < k\\
 0 & \text { if } j \ge k\end{cases}.\]

As in the previous section, let $T_{k}(x_{1}, x_{2}, \ldots) = \F_{2}[x_{1}, x_{2}, \ldots]/(x_{1}^{2^{k}}, x_{2}^{2^{k}}, \ldots)$.

Let $W_{1}$ denote the $E(Q_{1}, Q_{2})$-algebra 
\[ W_{1} = T_{1}(\bar{\xi}_{4},\bar{\xi}_{5},\ldots) \otimes \F_{2}[\bar{\xi}_{2}^{4}, \bar{\xi}_{3}^{4},\ldots ] \otimes \F_{2}[\bar{\xi}_{1}^{8}] .\] 

Let $W_{o}$ denote the $E(Q_{0}, Q_{2})$-algebra 
\[ T_{1}(\bar{\xi}_{5}, \bar{\xi}_{7}, \ldots) \otimes \F_{2}[\bar{\xi}_{4}^{2}, \bar{\xi}_{6}^{2}, \ldots ] \otimes \F_{2}[\bar{\xi}_{2}^{8}] ,\]

and let $W_{e}$ denote the $E(Q_{0}, Q_{2})$-algebra
\[ T_{1}(\bar{\xi}_{4}, \bar{\xi}_{6}, \ldots) \otimes \F_{2}[\bar{\xi}_{3}^{2}, \bar{\xi}_{5}^{2}, \ldots ] \otimes \F_{2}[\bar{\xi}_{1}^{8}] .\]

The following lemma is an adaptation of \cite[Lemma~2.31]{dominic_even}. It is the $2$-primary versions of \cref{lem:2.30}.

\begin{lemma}\label{marg:even:one}
As an $E(Q_{1}, Q_{2})$-module, 

\[ H_{*}BP\2 \cong W_{1} \otimes T_{3}(\bar{\xi}_{1}^{2}) \otimes  T_{1}(\bar{\xi_{2}}^{2}, \bar{\xi}_{3}^{2}, \ldots).\]

As an $E(Q_{0}, Q_{2})$-module, 
\[ H_{*}BP\2 \cong W_{e} \otimes W_{o} \otimes T_{3}(\bar{\xi}_{1}^{2}, \bar{\xi}_{2}^{2}).\]

\end{lemma}

\begin{lemma}\cite[Proposition~2.14]{dominic_even}\label{marg:even:two}
    At the prime $2$, the Margolis homology of $H_{*}BP\2$ is
\[ \M_{*}(H_{*}BP\2, Q_{0}) \cong \F_{2}[\bar{\xi}_{1}^{2}, \bar{\xi}_{2}^{2} ] \]
\[ \M_{*}(H_{*}BP\2, Q_{1}) \cong \F_{2}[\bar{\xi}_{1}^{2}] \otimes E(\bar{\xi}_{2}^{2}, \bar{\xi}_{3}^{2}, \ldots) \]
\[ \M_{*}(H_{*}BP\2, Q_{2}) \cong T_{2}(\bar{\xi}_{1}^{2}, \bar{\xi}_{2}^{2}, \ldots) .\]

\end{lemma}

Combining \cref{marg:even:one} and \cref{marg:even:two}, we can determine the Margolis homology of $W_{1}$, $W_{e}$, and $W_{o}$. The proof is completely analogous to that of \cref{marg:R1}. 

\begin{proposition}
    At the prime $2$, the Margolis homology of $W_{1}$ is
    \[ \M_{*}(W_{1}, Q_{1}) \cong \F_{p}[\bar{\xi}_{1}^{8}] \]
        \[ \M_{*}(W_{1}, Q_{2}) \cong T_{1}(\bar{\xi}_{2}^{4}, \bar{\xi}_{3}^{4}, \ldots) .\]

\end{proposition}

\begin{proposition}
    At the prime $2$, the Margolis homology of $W_{e}$ and $W_{o}$ is 
     \[ \M_{*}(W_{e}, Q_{0}) \cong \F_{p}[\bar{\xi}_{1}^{8}] \]
          \[ \M_{*}(W_{e}, Q_{2}) \cong T_{2}(\bar{\xi}_{3}^{2}, \bar{\xi}_{5}^{2}, \ldots) \]
        \[ \M_{*}(W_{o}, Q_{0}) \cong \F_{p}[\bar{\xi}_{2}^{8}] \]
        \[ \M_{*}(W_{o}, Q_{2}) \cong T_{2}(\bar{\xi}_{4}^{2}, \bar{\xi}_{6}^{2}, \ldots) .\]
\end{proposition}

Note that all monomials in $W_{1}$ have weight divisible by $8$. So we will define $W_{1}(k)$ to be the weight $8k$ component of $W_{1}$. Note also that $W_{e}$ consists solely of monomials of weight divisible by $8$ and $W_{o}$ of monomials divisible by 16, so we will let $W_{e}(k)$ be the weight $8k$ component and $W_{o}(k)$ the weight $16k$ component of $W_{o}$. Recall that the action of $E(2)$ on $H_{*}BP\2$ is weight-preserving, so we get a decomposition of the form
\[ \M_{*}(W_{1}, Q_{j}) \cong \bigoplus\limits_{n=0}^{\infty} \M_{*}(W_{1}(n), Q_{j}) \]
for $j=1,2$, as well as decompositions
\[ \M_{*}(W_{e}, Q_{j}) \cong \bigoplus\limits_{n=0}^{\infty} \M_{*}(W_{e}(n), Q_{j}) \]
\[ \M_{*}(W_{o}, Q_{j}) \cong \bigoplus\limits_{n=0}^{\infty} \M_{*}(W_{o}(n), Q_{j}) \]
for $j=0,2$.

Next we will compute the Margolis homology of these submodules. The computations are exactly analogous to those for the odd-primary versions (\cref{marg:w1:n}, \cref{marg:W:eo:n}).

\begin{proposition}
At the prime $2$, the Margolis homology of $W_{1}(n)$ is

\[ M_{*}(W_{1}(n), Q_{1}) \cong F_{2}\{\bar{\xi}_{1}^{8n} \}\]
\[ M_{*}(W_{1}(n), Q_{2}) \cong F_{2}\{ (\bar{\xi}_{2}^{n_{0} } \bar{\xi}_{3}^{n_{1} } \cdots \bar{\xi}_{r}^{n_{r} } )^{4} \} \]

where $n = n_{0} + 2n_{1} + \cdots + 2^{r}n_{r}$ is the $2$-adic decomposition of $n$. 
\end{proposition}

\begin{proposition}\label{marg:W:eo:n:pis2}
At the prime $2$, the $Q_{0}$-Margolis homologies of $W_{o}(n)$ and $W_{e}(n)$ are
\[ \M_{*}\big(W_{o}(n), Q_{0} \big) \cong
 \F_{p}\{ \bar{\xi}_{1}^{8n} \}  
\]

\[ \M_{*}\big(W_{e}(n), Q_{0} \big) \cong
\F_{p}\{ \bar{\xi}_{2}^{8n} \}
.\]

Suppose that $n \in \N$, with $p$-adic expansion
$n = n_{0} + n_{1}p + \cdots n_{r}p^{r}$. Let  $m_{i} = n_{2i} + pn_{2i+1}$, and $s = \lfloor r/2 \rfloor$. Then
\[ \M_{*}\big(W_{e}(n), Q_{2} \big) \cong
 \F_{p}\{ \bar{\xi}_{4}^{2m_{0}}\bar{\xi}_{6}^{2m_{1}} \cdots \bar{\xi}_{2s+4}^{2m_{s}} \}.
\]

Likewise,
\[ \M_{*}\big(W_{o}(n), Q_{2} \big) \cong
 \F_{p}\{ \bar{\xi}_{3}^{2m_{0}}\bar{\xi}_{5}^{2m_{1}} \cdots \bar{\xi}_{2s+3}^{2m_{s}} \}
.\] 

\end{proposition}

\subsubsection{The $Ext$ groups of $H_{*}BP\2$ on subalgebras of $E(2)$}\label{sec:ext:analysis}

Now we can return to working at an arbitrary prime $p$.

\begin{proposition}\label{ext:W:1}
For all $s > 0$ and $t-s$ odd, $Ext_{E(Q_{1}, Q_{2})}^{s,t}(\F_{p}, W_{1}) = 0$.
\end{proposition}

\begin{proof}
We will state the proof for odd primes. The argument at $p=2$ is exactly analogous. Recall from \cref{marg:decomp:R1} that
\[\M_{*}(W_{1}, Q_{j}) \cong \bigoplus\limits_{k=0}^{\infty} \M_{*}( W_{1}(k), Q_{j} ) \] 
for $j = 0,2$.
The Margolis homology of $W_{1}(n)$ is of the form
\[ \M_{*}(W_{1}(n), Q_{1}) \cong
\F_{p}\{ \bar{\xi}_{1}^{p^{2}n} \} 
.\]

\[ \M_{*}(W_{1}(n), Q_{2}) \cong
\F_{p}\{ (\bar{\xi}_{2}^{n_{0}}\bar{\xi}_{3}^{n_{1}}\cdots \bar{\xi}_{r+2}^{n_{r}})^{p} \} \]

where  $n_{0} + n_{1}p + \cdots + n_{r}p^{r}$ is the $p$-adic expansion of $k$. Note that $|\bar{\xi}_{1}^{p^{2}n}|$ is even, and $|(\bar{\xi}_{2}^{n_{0}}\bar{\xi}_{3}^{n_{1}}\cdots \bar{\xi}_{r+2}^{n_{r}})^{p}| > |\bar{\xi}_{1}^{p^{2}n}|$. So by \cref{marg:aux}, $Ext_{E(Q_{1}, Q_{2})}^{s,t}(\F_{p}, W_{1}(n))$ is concentrated in even $t-s$ degrees for all $n \in \N$ and $s > 0$. It follows that at odd primes, $Ext_{E(Q_{1}, Q_{2})}^{s,t}(\F_{p}, W_{1})$ is concentrated in even $t-s$ degrees for all $s > 0$. 
\end{proof}

We have an analogous result for $W_{e} \otimes W_{o}$. 

\begin{proposition}\label{ext:W:eo}
For all $s > 0$ and $t-s$ odd, $Ext_{E(Q_{0}, Q_{2})}^{s,t}(\F_{p}, W_{e} \otimes W_{o}) = 0$.
\end{proposition}

\begin{proof}
We will state the proof for odd primes. The argument at $p=2$ is exactly analogous. \cref{marg:decomp:Re} and \cref{marg:decomp:Ro} tells us that
\[ W_{e} \otimes W_{o} \cong \bigoplus\limits_{n,n' \in \N} W_{e}(n) \otimes W_{o}(n').\]

Consider a pair $n,n' \in \N$, with $p$-adic expansions $n = n_{0} + n_{1}p + \cdots + n_{r}p^{r}$ and $n' = n'_{0} + n'_{1}p + \cdots + n'_{r}p^{r'}$ \cref{marg:W:eo:n} tells us that the Margolis homology of $W_{e}(n) \otimes W_{e}(o)$ is
\[ \M_{*}\big(W_{e}(n) \otimes W_{o}(n'), Q_{0} \big) \cong
\F_{p}\{ \bar{\xi}_{2}^{p^{2}n} \otimes \bar{\xi}_{1}^{p^{2}n'} \}
\]

\[ \M_{*}\big(W_{e}(n) \otimes W_{o}(n') \big) \cong
 \F_{p}\{ \bar{\xi}_{4}^{m_{0}}\bar{\xi}_{6}^{m_{1}} \cdots \bar{\xi}_{2r+4}^{m_{r}} \otimes \bar{\xi}_{3}^{m'_{0}}\bar{\xi}_{5}^{m'_{1}} \cdots \bar{\xi}_{2r+3}^{m'_{r}} , \}
\]

where $m_{i} = n_{i} + pm_{i+1}$.
Note that $|\bar{\xi}_{2}^{p^{2}n} \otimes \bar{\xi}_{1}^{p^{2}n'} |$ is even, and \[|\bar{\xi}_{4}^{m_{0}}\bar{\xi}_{6}^{m_{1}} \cdots \bar{\xi}_{2r+4}^{m_{r}} \otimes \bar{\xi}_{3}^{m'_{0}}\bar{\xi}_{5}^{m'_{1}} \cdots \bar{\xi}_{2r+3}^{m'_{r}}| > |\bar{\xi}_{2}^{p^{2}n'} \otimes \bar{\xi}_{1}^{p^{2}n}  |.\]

So it follows from \cref{marg:aux} that $Ext_{E(Q_{0}, Q_{2})}^{s,t}\big(\F_{p}, W_{e}(n) \otimes W_{o}(n')\big) = 0$ for all $t-s$ odd and $s > 0$.
\end{proof} 

\begin{proposition} \label{ext:BP:0}
Let $0 \le j < h \le 2$. Then for all $s > 0$ and $t-s$ odd, \[Ext^{s,t}_{E(Q_{j}, Q_{h})}(\F_{p}, H_{*}BP\2) = 0.\]
\end{proposition}
\begin{proof}
We will state the proof for odd primes, but the arguments at $p=2$ are exactly analogous. The case $j=0$, $h=1$ follows from \cite[3.18]{dominic_odd}. 

Next we will consider the case $j=1, h=2$. Recall that \[H_{*}BP\2 \cong_{E(Q_{1}, Q_{2})} W_{1} \otimes T_{2}(\bar{\xi}_{1}) \otimes T_{1}(\bar{\xi}_{2}, \bar{\xi}_{3}, \ldots) .\]
Note that $E(Q_{1}, Q_{2})$ acts trivially on $T_{2}(\bar{\xi}_{1}) \otimes T_{1}(\bar{\xi}_{2}, \bar{\xi}_{3}, \ldots)$, so 
\[ H_{*}BP\2 \cong_{E(Q_{1}, Q_{2})} \bigoplus\limits_{x \in \mathcal{J}} \Sigma^{|x|}W_{1},\] 
where $\mathcal{J}$ is the set of all monomials in $T_{2}(\bar{\xi}_{1}) \otimes T_{1}(\bar{\xi}_{2}, \bar{\xi}_{3}, \ldots)$. Note that all monomials in $\mathcal{J}$ have even degree, so indeed by \cref{ext:W:1},
 $Ext^{s,t}_{E(Q_{1}, Q_{2})}(\F_{p}, H_{*}BP\2)$ is concentrated in even $t-s$ degrees for $s>0$. 
 
Finally we will check the $j=0, h=2$ case. Recall from \cref{lem:2.30} that \[H_{*}BP\2 \cong_{E(Q_{0}, Q_{2})} W_{o} \otimes W_{e} \otimes T_{2}(\bar{\xi}_{1}, \bar{\xi}_{2}).\]

Note that $E(Q_{0}, Q_{2})$ acts trivially on $T_{2}(\bar{\xi}_{1}, \bar{\xi}_{2})$, so 
\[ H_{*}BP\2 \cong_{E(Q_{0}, Q_{2})} \bigoplus\limits_{x \in \mathcal{I}} \Sigma^{|x|}W_{1},\] 
where $\mathcal{I}$ is the set of all monomials in $T_{2}(\bar{\xi}_{1}, \bar{\xi}_{2})$. Note that all monomials in $\mathcal{I}$ have even degree, so indeed by \cref{ext:W:eo},
 $Ext^{s,t}_{E(Q_{0}, Q_{2})}(\F_{p}, H_{*}BP\2)$ is concentrated in even $t-s$ degrees for $s>0$. 

\end{proof} 

\begin{theorem} \label{ext:total:C}
Let $(i,j,h)$ be any permutation of $(0,1,2)$. Then for all $t-s$ odd, \[Ext_{E(Q_{j}, Q_{h})}^{s,t}(\F_{p}, H_{*}^{BP\2}C) = 0.\]
\end{theorem}

\begin{proof}

Recall that $H_{*}BP\2 \cong_{E(2)} H_{*}^{BP\2}C \oplus \overline{H_{*}V}$, and that $\overline{H_{*}V}$ is a free $E(2)$-module. So 
\[Ext_{E(Q_{j}, Q_{h})}^{s,t}(\F_{p}, H_{*}^{BP\2}C) \cong Ext_{E(Q_{j}, Q_{h})}^{s,t}(\F_{p}, H_{*}BP\2) \]
for all $s > 0$. So by \cref{ext:BP:0}, $Ext_{E(Q_{j}, Q_{h})}^{s,t}(\F_{p}, H_{*}^{BP\2}C)$ is concentrated in even degrees for all $s > 0$. 
So all that is left to check is that there are no classes in odd $t$-degree on the $(s=0)$-line. 

Recall from \cref{split:Ri} that \[H_{*}^{BP\2}C \cong_{E(Q_{j}, Q_{h})} S_{i} \oplus R_{i} .\] 
\cref{ext:Si} tells us that $Ext^{s,t}_{E(Q_{j}, Q_{h})}(\F_{p}, S_{i})$ is concentrated in even $(t-s)$-degree. So all that is left is to show that $Ext^{0,t}_{E(Q_{j}, Q_{h})}(\F_{p}, R_{i})$ contains no classes in odd $t$-degree.

Suppose towards a contradiction that there was a nonzero class $x \in Ext^{0,t}_{E(Q_{j}, Q_{h})}(\F_{p}, R_{i})$ such that $|x|_{t}$ is odd. Note that $|v_{j}|_{t-s}$ and $|v_{h}|_{t-s}$ are even. Since we have already shown that there are no classes in odd $t-s$ degree above the $(s=0)$-line in $Ext_{E(2)_{*}}^{s,t}(\F_{p}, H_{*}^{BP\2}C)$, it follows that $v_{j}x = v_{h}x = 0$. Such a class could only come from a free summand in $R_{i}$. But we already know from \cref{split:Ri} that $R_{i}$ contains no free summands. So indeed there are no classes in odd $t$ degree on the $(s=0)$-line of $Ext_{E(2)_{*}}^{s,t}(\F_{p}, H_{*}^{BP\2}C)$.
\end{proof}

\subsubsection{Checking for hidden extensions}\label{sec:nhe}

\begin{theorem}\label{b:h:e}
There are no hidden extensions in the Adams spectral sequence
\[ Ext_{E(2)_{*}} ^{s,t}\big( \F_{p}, H_{*}BP\2 \big) \Longrightarrow BP\2_{t-s}BP\2.\]

\end{theorem}

\begin{proof}
Recall that $H_{*}BP\2 \cong_{E(2)} H_{*}^{BP\2}C \oplus H_{*}^{BP\2}V$. First we will check that there are no hidden extensions in $Ext_{E(2)_{*}}^{s,t}(\F_{p},H_{*}^{BP\2}C)$. Let $0 \le i < j \le 2$. \cref{ext:total:C} tells us that for all $t-s$ odd, \[Ext_{E(Q_{j}, Q_{h})}^{s,t}(\F_{p}, H_{*}^{BP\2}C) = 0.\] By \cref{bss:torsion}, this implies that there is no $v_{i}$-torsion in $Ext_{E(2)_{*}}(\F_{p},H_{*}^{BP\2}C)$. That is, if $x \in Ext_{E(2)_{*}}(\F_{p},H_{*}^{BP\2}C)$, then $v_{i}^{r}x \neq 0$ for $i=0,1,2$ and $r \in \N$. So indeed there are no hidden extensions in $Ext_{E(2)_{*}}^{s,t}(\F_{p},H_{*}^{BP\2}C)$.  Recall that $V$ is the Eilenberg-Maclane space for an $\F_{p}$-vector space, so there are also no hidden extensions in $Ext_{E(2)_{*}}^{s,t}(\F_{p},H_{*}^{BP\2}V)$. 
\end{proof}

We also want to show that there are no hidden extensions in the Adams spectral sequence 
\[ Ext_{E(2)_{*}} \big( \F_{p}, H_{*}l_{k} \big) \Longrightarrow BP\2_{*}l_{k}.\]

\begin{theorem}\label{l:h:e}
There are no hidden extensions in the Adams spectral sequence 
\[ Ext_{E(2)_{*}}^{s,t} \big( \F_{p}, H_{*}l_{k} \big) \Longrightarrow BP\2_{t-s}l_{k}.\] 
\end{theorem}

\begin{proof}
Recall from \cref{thm:bp split} that $BP\2 \smash l_{k}$ splits as 
\[ BP\2 \smash l_{k} \simeq C_{k} \vee V_{k}, \]
where $V_{k}$ is a sum of mod-$p$ Eilenberg-Maclane spectra and $C_{k}$ is $v_{2}$-torsion free. So the Adams spectral sequence converging to $BP\2_{*}l_{k}$ splits as the following two spectral sequences:
\[ Ext_{E(2)_{*}} \big( \F_{p}, H_{*}^{BP\2}C_{k} \big) \Longrightarrow \pi_{*}C_{k}\] 
\[ Ext_{E(2)_{*}} \big( \F_{p}, H_{*}^{BP\2}V_{k} \big) \Longrightarrow \pi_{*}V_{k}.\] 

Since $V_{k}$ is a sum of mod-$p$ Eilenberg-Maclane spectra, there are no hidden extensions in the spectral sequence converging to $\pi_{*}V_{k}$. All that is left to check is that there are no hidden extensions in the component converging to $\pi_{*}C_{k}$.

Recall from \cref{thm:dominic:splitting2} that 
\[H_{*}l_{k} \cong H_{*}^{BP\2}C_{k} \oplus H_{*}^{BP\2}V_{k},\]

where $H_{*}^{BP\2}V_{k}$ is a free $E(2)$-module and $H_{*}^{BP\2}C_{k}$ contains no free summands. 

Recall from \cref{eiso} that \[ H_{*}BP\2 \cong_{E(2)} \bigoplus\limits_{k=0}^{\infty} \Sigma^{qk} H_{*}l_{k} .\] 
Since $H_{*}^{BP\2}C_{k}$ and $H_{*}^{BP\2}C$ both contain no free summands (and $H_{*}^{BP\2}V_{k}$ and $H_{*}^{BP\2}V$ consist entirely of free summands), it follows that 
\[ H_{*}^{BP\2}C \cong \bigoplus_{k=0}^{\infty}H_{*}^{BP\2}C_{k}.\]

So \[ Ext_{E(2)_{*}}^{s,t} \big( \F_{p}, H_{*}^{BP\2}C \big) \cong \bigoplus_{k=0}^{\infty}Ext_{E(2)_{*}}^{s,t} \big( \F_{p}, H_{*}^{BP\2}C_{k} \big).  \]
It follows from \cref{ext:total:C} that there is no $v_{i}$-torsion in $Ext_{E(2)_{*}}^{*,*} \big( \F_{p}, H_{*}^{BP\2}C_{k} \big)$. So there is no room for hidden extensions in $Ext_{E(2)_{*}}^{*,*} \big( \F_{p}, H_{*}^{BP\2}C_{k} \big)$.
\end{proof}

\subsection{An isomorphism of $E_{2}$-pages}\label{sec:35}
\subsubsection{The Adams spectral sequence is multiplicative}
Next, we will show that the $E_{2}$-pages of the hypercohomology spectral sequence and universal coefficient spectral sequences are isomorphic in the cases that we are interested in. We will need to use the following multiplicative properties of the Adams spectral sequence.

Specifically, we will use the following theorem of Adams.

\begin{proposition}\cite[Thm~2.3.3]{green_book}\label{prop:rav_ass_mult}
Let $Y'$, $Y''$, and $Y$ be connective spectra whose $H\F_{p}$-homology is of finite type, with a pairing 

\[\alpha: Y^{'} \smash Y^{''} \rightarrow Y.\]

Then $\alpha$ induces a natural pairing on the Adams spectral sequences
\[\alpha: {E_{r}^{ASS}(S^{0},Y^{'})} \otimes {E_{r}^{ASS}(S^{0},Y^{''})} \Longrightarrow {E_{r}^{ASS}(S^{0},Y)}\]

with the following properties:
\begin{enumerate}
    \item The pairing on $E_{r}$ induces the pairing on $E_{r+1}$.
    \item The pairing on $E_{\infty}$ corresponds to 
    \[\alpha_{*}: \pi_{*}Y' \otimes \pi_{*}Y'' \rightarrow \pi_{*}Y.\]
    \item Let $c: Ext_{A_{*}}(\F_{p}; H_{*}Y')  \otimes Ext_{A_{*}}(\F_{p}, H_{*}Y'') \rightarrow Ext_{A_{*}}(\F_{p}, H_{*}Y'\otimes H_{*}Y'')$ denote the external cup product. Let $\beta$ be the composite
    \[\beta: H_{*}Y' \otimes H_{*}Y'' \rightarrow H_{*}(Y' \smash Y'') \overset{\alpha_{*}}{\rightarrow} E_{*}Y,\]
    
    where the first map is the K\"unneth isomorphism. Then the composite 
    
    \begin{equation} Ext_{A_{*}}(\F_{p}, H_{*}Y')  \otimes Ext_{A_{*}}(\F_{p}, H_{*}Y'') \overset{c}{\rightarrow} Ext_{A_{*}}(\F_{p}, H_{*}Y'\otimes H_{*}Y'') \overset{\beta_{*}}{\rightarrow } Ext_{A_{*}}(\F_{p}, H_{*}Y)\end{equation}\label{cup:comp}
    
    is the pairing induced on $E_{2}$ by $\alpha$.
    
\end{enumerate}
\end{proposition}

Let $\mu: BP\2 \smash BP\2 \rightarrow BP\2$ denote the usual multiplication. It follows that the $BP\2$-module action \[BP\2 \smash BP\2 \smash X \overset{\mu \smash 1}{\longrightarrow} BP\2 \smash X\] induces a natural pairing on the Adams spectral sequence 
\[ E_{r}(S^{0}, BP\2) \otimes E_{r}(S^{0}, BP\2 \smash X) \rightarrow E_{r}(S^{0}, BP\2 \smash X),  \]

such that the action on the $E_{2}$-terms is the composite described in (\ref{cup:comp}), and the action of the $E_{\infty}$-page corresponds to the action 
\[BP\2_{*} \otimes BP\2_{*}X \overset{\mu_{*} \otimes 1}{\longrightarrow} BP\2_{*}X .\]

Recall that the $E_{2}$-page of the Adams spectral sequence $E_{r}(S^{0}, BP\2)$ is isomorphic to $Ext_{E(2)}(\F_{p}, \F_{p})$, and that the $E_{2}$-page of the Adams spectral sequence $E_{r}(S^{0}, BP\2\smash X)$ is isomorphic to $Ext_{E(2)}(\F_{p}, H_{*}X)$. We will use the following property of the pairing discussed above.  
\begin{proposition}
Let $X = l_{k}$ or $BP\2$. Consider the Adams spectral sequences
\[Ext_{E(2)}(\F_{p}, \F_{p}) \Longrightarrow BP\2_{*} \]
\[Ext_{E(2)}(\F_{p}, H_{*}X) \Longrightarrow BP\2_{*}X .\]
The action on the $E_{2}$-page
\[ Ext_{E(2)}(\F_{p}, \F_{p}) \otimes Ext_{E(2)}(\F_{p}, H_{*}X) \rightarrow Ext_{E(2)}(\F_{p}, H_{*}X) \]
determines the action 
\[ BP\2_{*} \otimes BP\2_{*}X \rightarrow BP\2_{*}X \]
up to hidden extensions.
\end{proposition}

\subsubsection{Comparing the associated graded modules as vector spaces.}

Consider the Adams spectral sequence \[Ext_{E(2)}(\F_{p}, H_{*}BP\2) \Longrightarrow BP\2_{*}.\]
The associated graded algebra to the Adams filtration on $BP\2_{*}$ is isomorphic to $\P(2)$. Let $gr_{*}: Mod_{BP\2_{*}} \rightarrow Mod_{\P(2)}$ denote the functor sending a module over $BP\2_{*}$ to its associated graded module over $\P(2)$.\\

\begin{proposition}
Let $gr BP\2_{*}X$ denote the associated graded module to the Adams filtration on $BP\2_{*}X$. If $X = BP\2$ or $l_{k}$, then there exists an $\F_{p}$-vector space isomorphism 
\[\varphi: Ext_{E(2)_{*}} ( \F_{p}, H_{*}X) \rightarrow BP\2_{*}X .\] 
\end{proposition}

\begin{proof}
Recall from \cref{thm:dominic:splitting2} that the Adams spectral sequences
\[ Ext_{E(2)_{*}}(\F_{p}, H_{*}BP\2) \Longrightarrow BP\2_{*}BP\2 \]
\[ Ext_{E(2)_{*}}(\F_{p}, H_{*}l_{k}) \Longrightarrow BP\2_{*}l_{k} \]

collapse at the $E_{2}$-page. So each nonzero class $x \in E_{2}^{*,*}$ corresponds to a unique nonzero class $\widetilde{x} \in grBP\2_{*}X$, and vice versa. 
\end{proof}

\subsubsection{Comparing the multiplicative structures.}

\begin{proposition}
Let $x,y \in Ext_{E(2)_{*}}(\F_{p}, H_{*}X) $. Let $\widetilde{x}$, $\widetilde{y}$ denote classes in $BP\2_{*}X$ such that $gr(\widetilde{x}) = \varphi(x)$ and $gr(\widetilde{y}) = \varphi(y)$. Then $p^{i_{0}}v_{1}^{i_{1}}v_{2}^{i_{2}}\widetilde{x} = \widetilde{y}$ if and only if $v_{0}^{i_{0}}v_{1}^{i_{1}}v_{2}^{i_{2}}x = y$ .
\end{proposition}

\begin{proof}
As remarked above, these spectral sequences are multiplicative. So the $\P(2) \cong Ext_{E(2)_{*}}(\F_{p}, \F_{p})$-module action on the $E_{2}$-page determines the $BP\2_{*}$-action on the target, up to hidden extensions. It is well-known that the $v_{0}$-action extends so that if $v_{0}x = y$, then $p\widetilde{x} = \widetilde{y}$. Furthermore, we have shown in \cref{l:h:e} and \cref{b:h:e} that there is no room for hidden extensions in either of these spectral sequences. 
So when $X = BP\2$ or $l_{k}$, the $BP\2_{*}$-module multiplication on $BP\2_{*}X$ is completely determined by the $\P(2)$-module multiplication on $Ext_{E(2)_{*}}(\F_{p}, H_{*}X)$. That is, for any $x,y \in E_{2}^{s,t}$, $v_{0}^{i}v_{1}^{j}v_{2}^{k}x = y$ if and only if $p^{i}v_{1}^{j}v_{2}^{k}\widetilde{x} = \widetilde{y}$. 
\end{proof}

Note that $gr(BP\2_{*}) \cong \P(2)$, so we can define a functor $gr_{*}: Mod_{BP\2_{*}} \rightarrow Mod_{\P(2)}$ which sends a $BP\2_{*}$-module to its associated graded module.

\begin{proposition}\label{gr:equiv}
The functor $gr_{*}: Mod_{BP\2_{*}} \rightarrow Mod_{\P(2)}$ restricts to an equivalence of categories on finitely generated free modules. 
\end{proposition}

\begin{proof}
First we confirm that $gr_{*}$ is essentially surjective on the subcategory of finitely generated free $\P(2)$-modules. Any finitely generated free $\P(2)$-module can be written in the form $\bigoplus\limits_{i \in I} \P(2)\{x_{i}\}$, where $I$ is a finite set. The module $\bigoplus\limits_{i \in I} BP\2_{*}\{x_{i}\}$ is also a free finitely generated $BP\2_{*}$-module, and \[gr_{*}\Big(\bigoplus\limits_{i \in I} BP\2_{*}\{x_{i}\}\Big) = \bigoplus\limits_{i \in I} \P(2)\{x_{i}\}.\] So indeed $gr_{*}$ is essentially surjective on the subcategory of finitely generated free $\P(2)$-modules.

Now we will check that $gr_{*}$ is fully faithful on the subcategory of finitely generated free $\P(2)$-modules. First, note that 
\[ Hom_{BP\2_{*}}(BP\2_{*}, BP\2_{*}) \cong Hom_{\P(2)}(gr_{*}BP\2_{*}, gr_{*}BP\2_{*}) .\]
It follows that \[ Hom_{BP\2_{*}}\big( F, G \big) \cong Hom_{\P(2)}\big( gr_{*}F, gr_{*}G \big)   \]
for all free and finitely generated $BP\2_{*}$-modules $F, G$.
\end{proof}

Now we are ready to prove the following theorem.
\begin{lemma}\label{leftiso}
There exists an isomorphism
\[ Ext_{\P(2)}^{u,*}\big(  Ext_{E(2)_{*}} ( \F_{p}, H_{*}l_{k}), Ext_{E(2)_{*}} ( \F_{p}, H_{*}BP\2) \big) \cong Ext_{BP\2_{*}}^{u,*}\big(BP\2_{*}l_{k}, BP\2_{*}BP\2  \big) .\]
\end{lemma}

\begin{proof}

Consider the module $BP\2_{*}l_{k}$. We can construct a projective resolution $P^{\bullet}$ over $BP\2_{*}$ of the following form

\[
\begin{tikzcd}[ampersand replacement=\&]
0
\&
BP\2_{*}l_{k}  \arrow[l] 
\&
 \bigoplus\limits_{j}\Sigma^{r_{0}(j)}BP\2_{*} \arrow[l,"\epsilon"] 
\& 
\cdots \arrow[l, "d_{0}"] 
\&
\bigoplus\limits_{j}\Sigma^{r_{3}(j)}BP\2_{*} \arrow[l, "d_{3}"]
\& 
0 \arrow[l] \\
\end{tikzcd}
.\]

Consider the chain complex $Q^{\bullet}$ obtained by applying $gr_{*}: Mod_{BP\2_{*}} \rightarrow Mod_{\P(2)}$ to $P^{\bullet}$. By \cref{gr:equiv}, $Q^{\bullet}$ is a projective resolution of $Ext_{E(2)_{*}}(\F_{p}, H_{*}l_{k})$. Furthermore, \\ $Hom_{\P(2)}(P^{\bullet}, BP\2_{*}BP\2)$ is quasi-isomorphic to $Hom_{E(2)}\big(Q^{\bullet}, Ext_{E(2)_{*}}(\F_{p}, H_{*}BP\2)\big)$. So indeed the $Ext$ terms are isomorphic.

\end{proof}
So we have constructed the following square.
\begin{equation}\label{bigsquare}
\begin{tikzcd}[ampersand replacement=\&]
Ext_{\P(2)}\big( Ext_{E(2)_{*}} ( \F_{p}, H_{*}l_{k}), Ext_{E(2)_{*}} ( \F_{p}, H_{*}BP\2) \big) \arrow[Rightarrow, r, "HSS"] \arrow[d, phantom, "\visom"]
\&
Ext_{E(2)}\big(H_{*}l_{k}, H_{*}BP\2  \big)\arrow[d, Rightarrow, "ASS"]\\
Ext_{BP\2_{*}}\big(BP\2_{*}l_{k}, BP\2_{*} BP\2 \big)\arrow[r, Rightarrow, "UCSS"]
\&
\left[  l_{k}, BP\2 \smash BP\2 \right] \\
\end{tikzcd}
\end{equation}

\section{Analyzing the potential obstructions}\label{sec:analysis:po}
Recall that a class $x \in Ext_{E(2)_{*}}\big( H_{*}l_{k}, H_{*}BP\2 \big)$ is said to be a potential obstruction (to lifting the map $\theta_{k}$) if $x$ has degree $|x|_{s,t}$ such that $s\ge 2$ and $(t-s)$ is odd. In this section, we will show that all these potential obstructions must generate a nonzero class on the $E_{\infty}$-page of the spectral sequence. This implies that none of the potential obstructions can be boundaries for any differential $d_{r}$, and so $d_{r}(\theta_{k}) = 0$ for all $r$. Thus $\theta_{k}$ survives the Adams spectral sequence as well, allowing us to lift the necessary maps to construct the splitting of \cref{thm:main}.

Recall the Adams spectral sequence (\ref{main:relASS}) \[ Ext_{E(2)_{*}}^{s,t}\big(H_{*}l_{k}, H_{*}BP\2\big) \Longrightarrow [BP\2 \smash l_{k}, BP\2 \smash BP\2]^{BP\2}.\] 

Likewise, we can use the the $BP\2$-module splitting of $BP\2 \smash l_{k}$ constructed in \cref{thm:bp split} to further divide this spectral sequence into a summand of two more spectral sequences. 

\begin{proposition}
The Adams spectral sequence 
\[ Ext_{E(2)_{*}}^{s,t}\big(H_{*}l_{k}, H_{*}BP\2\big) \Longrightarrow [BP\2 \smash l_{k}, BP\2 \smash BP\2]^{BP\2}\] 
splits as a sum of the following four spectral sequences: 

\[Ext_{E(2)_{*}} \big( H_{*}^{BP\2}C_{k}, H_{*}^{BP\2}C \big) \Longrightarrow [C_{k}, C]^{BP\2}  \]

\[Ext_{E(2)_{*}} \big( H_{*}^{BP\2}C_{k}, H_{*}^{BP\2}V \big) \Longrightarrow [C_{k}, V]^{BP\2}  \]

\[Ext_{E(2)_{*}} \big( H_{*}^{BP\2}V_{k}, H_{*}^{BP\2}V \big) \Longrightarrow [V_{k}, V]^{BP\2}  \]

\[Ext_{E(2)_{*}} \big( H_{*}^{BP\2}V_{k}, H_{*}^{BP\2}C \big) \Longrightarrow [V_{k}, C]^{BP\2}  .\]

Furthermore, any potential obstructions to lifting $\theta_{k}$ will be contained in the summand \[Ext_{E(2)_{*}} \big( H_{*}^{BP\2}C_{k}, H_{*}^{BP\2}C \big) .\]

\end{proposition}

\begin{proof}
The $BP\2$-module splittings of $BP\2 \smash BP\2$ and $BP\2 \smash l_{k}$ constructed in \cref{thm:bp split} allows us to split this Adams spectral sequence as a sum of four separate Adams spectral sequences: \[Ext_{E(2)_{*}} \big( H_{*}^{BP\2}C_{k}, H_{*}^{BP\2}C \big) \Longrightarrow [C_{k}, C]^{BP\2}  \]

\[Ext_{E(2)_{*}} \big( H_{*}^{BP\2}C_{k}, H_{*}^{BP\2}V \big) \Longrightarrow [C_{k}, V]^{BP\2}  \]

\[Ext_{E(2)_{*}} \big( H_{*}^{BP\2}V_{k}, H_{*}^{BP\2}V \big) \Longrightarrow [V_{k}, V]^{BP\2}  \]

\[Ext_{E(2)_{*}} \big( H_{*}^{BP\2}V_{k}, H_{*}^{BP\2}C \big) \Longrightarrow [V_{k}, C]^{BP\2}  .\]

Recall from \cref{prop:free} that $H_{*}^{BP\2}V$ and $H_{*}^{BP\2}V_{k}$ are both free $E(2)_{*}$-comodules of finite type. \cref{free:proj:inj} tells us that $H_{*}^{BP\2}V$ is therefore also injective, and $H_{*}^{BP\2}V_{k}$ is also projective. So each of the summands, except for $Ext_{E(2)_{*}}^{s,t}(H_{*}C_{k}, H_{*}^{BP\2}C)$, is concentrated on the $(s=0)$-line. Recall that any potential obstructions to lifting the map $\theta_{k}:BP\2 \smash l_{k}\rightarrow BP\2 \smash BP\2$ must have $s$-degree at least $2$. So to analyze the potential obstructions, we can restrict our attention to the summand \begin{equation}\label{Cass}Ext_{E(2)_{*}}^{s,t}(H_{*}C_{k}, H_{*}^{BP\2}C) \Longrightarrow [C_{k}, C]^{BP\2}.\end{equation}\\
\end{proof}

Since the splittings of $BP\2 \smash BP \2$ and $BP\2 \smash l_{k}$ that were used above are $BP\2$-module splittings, we can apply the same splittings to the universal coefficient spectral sequence. So instead of needing the entire square (\ref{bigsquare}), we can restrict our attention to the square below in analyzing the potential obstructions.

\begin{tikzcd}[ampersand replacement=\&]\label{smallsquare}
Ext_{\P(2)}\big( Ext_{E(2)_{*}} ( \F_{p}, H_{*}^{BP\2}C_{k}), Ext_{E(2)_{*}} ( \F_{p}, H_{*}^{BP\2}C) \big) \arrow[Rightarrow, r, "HSS"] \arrow[d, phantom, "\visom"]
\&
Ext_{E(2)}\big(H_{*}^{BP\2}C_{k}, H_{*}^{BP\2}C  \big)\arrow[Rightarrow, d, "ASS"]\\
Ext_{BP\2_{*}}\big(\pi_{*}C_{k}, \pi_{*}C \big)\arrow[Rightarrow, r, "UCSS"]
\&
\left[C_{k}, C\right]^{BP\2}\\
\end{tikzcd}
Now we will analyze the $E_{2}$-page $Ext_{BP\2_{*}}\big(\pi_{*}C_{k}, \pi_{*}C \big)$ of the universal coefficient spectral sequence. We will need to use a classical result of Auslander-Buchsbaum about Noetherian local rings and their modules.

\begin{definition}
Let $R$ be a commutative ring, $I$ an ideal in $R$, and $M$ an $R$-module. The $I$-depth of $M$ is
\[ depth_{I}(M) = min\{s| Ext_{R}(R/I, M) \neq 0 \} .\]

If $R$ has a maximal ideal $\mathfrak{m}$, then the depth of $R$ is defined as \[depth(R) = depth_{\mathfrak{m}}(R).\]
\end{definition}

The depth of $R$ is at most equal to the dimension of the ring $R$. When $depth(R) = dim(R)$, the ring is said to be Cohen-Macaulay.
\begin{proposition}\cite[p.451]{eisenbud2013commutative}
Let $R$ be a regular local ring. Then $R$ is Cohen-Macaulay.
\end{proposition}

\begin{lemma}[Auslander-Buchsbaum~formula]\cite[Thm~19.9]{eisenbud2013commutative}\label{auslander-buchsbaum}
Let $R$ be a Noetherian local ring with maximal ideal $I$. If $M$ is a nonzero finitely generated $R$-module with finite projective dimension $pd(M)$, then 
$pd(M) + depth(M) = depth(R) $.
\end{lemma}

This formula yields a vanishing line on the $E_{2}$-page of the hypercohomology spectral sequence.

\begin{proposition}\label{prop:zero_line}
For all $u > 2$,
\[  Ext_{\P(2)}^{u,*}\big( Ext_{E(2)_{*}} ( \F_{p}, H_{*}^{BP\2}C_{k}), Ext_{E(2)_{*}} ( \F_{p}, H_{*}^{BP\2}C) \big) = 0. \]
\end{proposition}

\begin{proof}

Note that $\P(2)$ is a regular local ring of dimension $3$ with maximal ideal $(v_{0}, v_{1}, v_{2})$, and so for any finitely generated $E(2)$-module $M$ over $E(2)$,

 \[depth(M) = 3 - pd(M)  .\] 
Culver's inductive computations in \cite[Section~3.3]{dominic_even}\cite[Section~5.1]{dominic_odd} tell us that the $\P(2)$-module $Ext_{E(2)_{*}}(\F_{p}, H_{*}l_{k})$ is finitely generated, so its submodule $Ext_{E(2)}(\F_{p}, H_{*}^{BP\2}C_{k} )$ must be finitely generated as well. So we can apply the Auslander-Buchsbaum formula (\cref{auslander-buchsbaum}). Note that if $M$ is $v_{2}$-torsion free, then $Hom_{\P(2)}(\F_{p}, M) = 0$. Recall from \cref{thm:dominic:splitting2} that $H_{*}^{BP\2}C_{k}$ is $v_{2}$-torsion free. So 
\[depth\big( Ext_{E(2)}(\F_{p}, H_{*}^{BP\2}C_{k} ) \big)\ge 1.\] Thus the projective dimension of $Ext_{E(2)_{*}}(\F_{p}, H_{*}^{BP\2}C_{k})$ over $\P(2)$ is at most $2$.

So indeed \[  Ext_{\P(2)}^{q}\big( Ext_{E(2)_{*}} ( \F_{p}, H_{*}^{BP\2}C_{k}), Ext_{E(2)_{*}} ( \F_{p}, H_{*}^{BP\2}C) \big) = 0 \text{ for all } q > 2 .\]
\end{proof}

Since the universal coefficient spectral sequence and hypercohomology spectral sequence have isomorphic $E_{2}$-pages, we can use this result to show that certain classes must survive the universal coefficient spectral sequence. 

\begin{proposition}
Let $x \in Ext_{BP\2_{*}}^{1,*}(\pi_{*}C_{k}, \pi_{*}C)$. Then $x$ survives the universal coefficient spectral sequence \[ ^{UCSS}E_{2}^{*,*}(C_{k}, \C) = Ext_{BP\2_{*}}^{1,*}(\pi_{*}C_{k}, \pi_{*}C) \Longrightarrow [C_{k}, C]^{BP\2}.\]
\end{proposition}
\begin{proof}
Recall from \cref{leftiso} that \[^{HSS}E_{2}(H_{*}^{BP\2}C_{k}, H_{*}^{BP\2}C) \cong {}^{UCSS}E_{2}(C_{k}, C) .\] So by \cref{prop:zero_line}, \[^{UCSS}E_{2}^{u,*}(C_{k}, C) = 0 \text{ for all $u>2$.}\]

The differential $^{UCSS}E_{r}^{u,*}: {}^{UCSS}E_{r}^{u,*} \rightarrow {}^{UCSS}E_{u}^{u+r,*}$ increases the $u$-degree of the differential by $r$. So $x$ is neither the target nor the source of a differential in the universal coefficient spectral sequence.
\end{proof}

Now we are ready to compare the $E_{2}$-page of the Adams spectral sequence to the $E_{2}$-page of the universal coefficient spectral sequence. 

\begin{proposition}\label{propiso}
The odd $(t-s)$-degree component of the $E_{2}$-page of the Adams spectral sequence is isomorphic to the $(u=1)$-line of the universal coefficient spectral sequence, that is,
\[ \bigoplus\limits_{t-s \text{ odd}} Ext_{E(2)_{*}}^{s,t} \big(  H_{*}^{BP\2}C_{k}, H_{*}^{BP\2}C \big) \cong Ext_{BP\2_{*}}^{u=1,*}(\pi_{*}C_{k}, \pi_{*}C). \]
\end{proposition}

\begin{proof}

Consider the hypercohomology spectral sequence 
\[ {}^{HSS}E_{2}^{u,r,t} \big(  H_{*}^{BP\2}C_{k}, H_{*}^{BP\2}C \big) \cong \bigoplus \limits_{ \substack{r= r_{2} - r_{1} \\ t = t_{2} - t_{1} } } Ext_{\P(2)}^{u}\big( Ext_{E(2)}^{r_{1}, t_{1}}( \F_{p},  H_{*}^{BP\2}C ), Ext_{E(2)}^{r_{2}, t_{2}}( \F_{p}, H_{*}^{BP\2}C ) \big)\] \[ {\Longrightarrow} Ext_{E(2)}^{ r-u, t }( H_{*}^{BP\2}C_{k}, H_{*}^{BP\2}C). \]

Let $x \in Ext_{E(2)_{*}}^{s,t} \big(  H_{*}^{BP\2}C_{k}, H_{*}^{BP\2}C \big)$, and let $\overline{x}$ denote an element that detects $x$ in  $^{HSS}E_{2}^{u,*} \big(  H_{*}^{BP\2}C_{k}, H_{*}^{BP\2}C \big)$. By \cref{thm:dominic:splitting2}, both $Ext_{E(2)_{*}}^{r_{1},t_{1} } ( \F_{p}, H_{*}^{BP\2}C_{k})$ and $Ext_{E(2)_{*}}^{r_{2},t_{2} } ( \F_{p}, H_{*}^{BP\2}C)$ are concentrated in degrees such that $t_{2}-s_{2}$ and $t_{1}-s_{1}$ are even. So $|\overline{x}|_{t-s}$ is odd if and only if $|x|_{u}$ is odd. \\

By \cref{prop:zero_line}, $^{HSS}E_{2}^{u,*} = 0$ for all $u > 2$. So if $|x|_{t-s}$ is odd, then $x$ must lift to the $(u=1)$-line of $^{HSS}E_{2}^{u,*} \big(  H_{*}^{BP\2}C_{k}, H_{*}^{BP\2}C \big)$. Furthermore, there is no room for differentials to leave or enter the $(u=1)$-line. Since these classes are all located in the same filtration, there is no room for hidden extensions. So 

\[ \bigoplus\limits_{t-s \text{ odd}} Ext_{E(2)_{*}}^{s,t} \big(  H_{*}^{BP\2}C_{k}, H_{*}^{BP\2}C \big) \cong {}^{HSS}E_{2}^{s,t} \big(  H_{*}^{BP\2}C_{k}, H_{*}^{BP\2}C \big) . \]

Furthermore, recall from \cref{leftiso} that 
\[ {}^{HSS}E_{2}^{s,t} \big(  H_{*}^{BP\2}C_{k}, H_{*}^{BP\2}C \big) \cong {}^{UCSS}E_{2}^{u=1,*}(C_{k}, C) .\]

So indeed
\[ \bigoplus\limits_{t-s \text{ odd}} Ext_{E(2)_{*}}^{s,t} \big(  H_{*}^{BP\2}C_{k}, H_{*}^{BP\2}C \big) \cong {}^{UCSS}E_{2}^{u=1,*}(C_{k}, C) . \]
\end{proof}

Recall that the potential obstructions in the Adams spectral sequence (\ref{main:relASS}) to lifting $\theta_{k}$ are all located in odd $(t-s)$ degree, so proving the following theorem will demonstrate that the potential obstructions survive.
\oddsurvives

\begin{proof}

By \cref{propiso}, any nonzero class $y$ on the $(u=1)$-line of the universal coefficient spectral sequence $^{UCSS}E_{2}^{u,*}\big(C_{k}, C \big)$ must lift to a nonzero class $\widetilde{y}$ in $[C_{k}, C]^{BP\2}$. Recall that $\pi_{*}C_{k}$ and $\pi_{*}C$ are concentrated in even degree, so $\widetilde{y}$ will have odd $t-s$ degree. So $\widetilde{y}$ must be detected by a nonzero class $[\widetilde{y}] \in Ext_{E(2)}^{s,t}\big(H_{*}^{BP\2}C_{k}, H_{*}^{BP\2}C  \big)$ having odd $t-s$ degree. \cref{propiso} tells us that there is at most one nonzero odd-degree class in $^{ASS}E_{2}^{*,*}$ for each nonzero class in $^{UCSS}E_{2}^{1,*}$. So the entire summand 

\[ \bigoplus\limits_{t-s \text{ odd}} Ext_{E(2)_{*}}^{s,t} \big(  H_{*}^{BP\2}C_{k}, H_{*}^{BP\2}C \big) \]

survives the spectral sequence.

\end{proof}

\posurvives

It follows that $d_{r}(\theta_{k}) = 0$ for all $k$. So we can indeed lift the family of maps $\{ \theta_{k}: H_{*}l_{k} \rightarrow H_{*}BP\2 \}$ to maps of spectra $\{ \widetilde{\theta_{k}}: BP\2 \smash l_{k} \rightarrow BP\2 \smash BP\2 \}$, and we have arrived at the following theorem. 

\begin{theorem}[Main Theorem]\label{thm:main}
For all primes $p\ge 5$, there exists a splitting 
\[BP\langle 2 \rangle \smash BP\langle 2 \rangle \simeq \bigvee_{k=0}^{\infty} \Sigma^{qk}BP\langle 2 \rangle \smash l_{k}.\]
\end{theorem}



%
\appendix

%
\bibliographystyle{alpha}
\bibliography{references}

\end{document}